\documentclass[11pt,reqno]{amsart}

\usepackage{amssymb, amsmath, amsthm}
\usepackage{hyperref}
\usepackage[alphabetic,lite]{amsrefs}
\usepackage{verbatim}
\usepackage{amscd}   
\usepackage[all]{xy} 
\usepackage{youngtab} 
\usepackage{young} 
\usepackage{ytableau}
\usepackage{tikz}
\usepackage{ mathrsfs }
\usepackage{cases}
\usepackage{array}
\usepackage{tabu}
\usepackage{calligra,mathrsfs}

\newcommand{\defi}[1]{{\upshape\sffamily #1}}

\renewcommand{\a}{\alpha}
\renewcommand{\b}{\beta}

\newcommand{\bw}{\bigwedge}
\renewcommand{\det}{\textrm{det}}

\renewcommand{\ll}{\lambda}

\newcommand{\oo}{\otimes}
\newcommand{\pd}{\partial}

\newcommand{\lie}{\mathfrak{g}}

\renewcommand{\S}{\mathbb{S}}

\newcommand{\D}{\mathcal{D}}

\newcommand{\I}{\mathcal{I}}

\newcommand{\Q}{\mathcal{Q}}

\newcommand{\sI}{\mathscr{I}}
\newcommand{\sP}{\mathscr{P}}
\newcommand{\sS}{\mathscr{S}}
\newcommand{\sV}{\mathscr{V}}

\DeclareMathOperator{\ShHom}{\mathscr{H}\text{\kern -3pt {\calligra\large om}}\,}
\DeclareMathOperator{\ShExt}{\mathscr{E}\text{\kern -2pt {\calligra\large xt}}\,}

\newcommand{\Aut}{\operatorname{Aut}}

\newcommand{\Ext}{\operatorname{Ext}}
\newcommand{\GL}{\operatorname{GL}}
\newcommand{\Hom}{\operatorname{Hom}}

\newcommand{\SL}{\operatorname{SL}}

\newcommand{\Sym}{\operatorname{Sym}}

\newcommand{\coker}{\operatorname{coker}}
\renewcommand{\det}{\operatorname{det}}

\renewcommand{\ker}{\operatorname{ker}}
\newcommand{\opmod}{\operatorname{mod}}
\newcommand{\opMod}{\operatorname{Mod}}

\newcommand{\rep}{\operatorname{rep}}

\newcommand{\bb}[1]{\mathbb{#1}}

\renewcommand{\rm}[1]{\textrm{#1}}
\newcommand{\mc}[1]{\mathcal{#1}}

\newcommand{\ol}[1]{\overline{#1}}

\newcommand{\tl}[1]{\tilde{#1}}

\def\lra{\longrightarrow}

\def\llra{\longleftrightarrow}
\def\hra{\hookrightarrow}

\newtheorem{theorem}{Theorem}[section]
\newtheorem*{theorem*}{Theorem}
\newtheorem*{problem*}{Problem}
\newtheorem{lemma}[theorem]{Lemma}

\newtheorem{corollary}[theorem]{Corollary}
\newtheorem*{corollary*}{Corollary}

\newtheorem*{thm-simples*}{Theorem on Simple $\D$-modules}
\newtheorem*{thm-category*}{Theorem on the Category of Equivariant $\D$-modules}
\newtheorem*{thm-iterated*}{Theorem on Iterated Local Cohomology Modules}

\theoremstyle{definition}

\newtheorem*{definition*}{Definition}
\newtheorem{example}[theorem]{Example}

\theoremstyle{remark}
\newtheorem{remark}[theorem]{Remark}
\newtheorem*{remark*}{Remark}

\numberwithin{equation}{section}

\begin{document}

\title{Equivariant $\D$-modules on binary cubic forms}

\author{Andr\'as C. L\H{o}rincz}
\address{Department of Mathematics, Purdue University, West Lafayette, IN 47907}
\email{alorincz@purdue.edu}

\author{Claudiu Raicu}
\address{Department of Mathematics, University of Notre Dame, 255 Hurley, Notre Dame, IN 46556\newline
\indent Institute of Mathematics ``Simion Stoilow'' of the Romanian Academy}
\email{craicu@nd.edu}

\author{Jerzy Weyman}
\address{Department of Mathematics, University of Connecticut, Storrs, CT 06269}
\email{jerzy.weyman@uconn.edu}


\date{\today}


\begin{abstract} 
 We consider the space $X=\Sym^3\bb{C}^2$ of binary cubic forms, equipped with the natural action of the group $\GL_2$ of invertible linear transformations of $\bb{C}^2$. We describe explicitly the category of $\GL_2$-equivariant coherent $\D_X$-modules as the category of representations of a quiver with relations. We show moreover that this quiver is of tame representation type and we classify its indecomposable representations. We also give a construction of the simple equivariant $\D_X$-modules (of which there are $14$), and give formulas for the characters of their underlying $\GL_2$-representations. We conclude the article with an explicit calculation of (iterated) local cohomology groups with supports given by orbit closures.
\end{abstract}

\maketitle

\section{Introduction}\label{sec:intro}

Let $W=\bb{C}^2$ denote a $2$-dimensional complex vector space, and let $X = \Sym^3 W$ denote the corresponding space of binary cubic forms. There is a natural action of the group $\GL(W) = \GL_2(\bb{C})$ (which we'll simply denote by $\GL$) on the space $X$, with four orbits:
\begin{itemize}
 \item The zero orbit $O_0 = \{0\}$.
 \item The orbit $O_2 = \{ l^3 : 0\neq l \in W\}$ of cubes of linear forms, whose closure $\ol{O_2}$ is the affine cone over the twisted cubic curve.
 \item The orbit $O_3 = \{ l_1^2 \cdot l_2 : 0 \neq l_1,l_2 \in W \mbox{ distinct up to scaling}\}$ whose closure $\ol{{O}_3}$ is the hypersurface defined by the vanishing of the cubic discriminant (and it is also the affine cone over the tangential variety to the twisted cubic curve).
 \item The dense orbit $O_4 = \{ l_1 \cdot l_2 \cdot l_3 : 0\neq l_1,l_2,l_3 \in W \mbox{ distinct up to scaling}\}$.
\end{itemize}
Note that our indexing of the orbits is chosen so that $O_i$ has dimension $i$ for $i=0,2,3,4$. Letting $\D = \D_X$ denote the Weyl algebra of differential operators on $X$ with polynomial coefficients, our goal is to completely classify and give concrete constructions of the simple $\GL$-equivariant $\D$-modules, as well as to give an explicit quiver description of the category of equivariant coherent $\D$-modules and to analyze its structure.

By the Riemann--Hilbert correspondence, the simple equivariant $\D$-modules are known to be indexed by irreducible equivariant local systems on the orbits of the group action, but their explicit realization is in general difficult to obtain (see Open Problem 3 in \cite[Section~6]{mac-vil}). In the case of binary cubic forms the classification is summarized as follows. 

\begin{thm-simples*}
 There exist 14 simple $\GL$-equivariant $\D$-modules on $X=\Sym^3 W$, which can be classified according to their support as follows:
 \begin{itemize}
  \item $6$ of the modules, denoted $G_{-1},G_0=S,G_1,G_2,G_3,G_4$, have full support and correspond to irreducible rank one equivariant local systems on $O_4$.
  \item $3$ of the modules, denoted $Q_0,Q_1,Q_2$, have full support and correspond to irreducible rank two equivariant local systems on $O_4$.
  \item one of the modules, denoted $P$, has support $\ol{O_3}$.
  \item $3$ of the modules, denoted $D_0,D_1,D_2$, have support $\ol{O_2}$.
  \item one of the modules, denoted $E$, has support $O_0$.
 \end{itemize}
 All the modules have a concrete description, starting with $G_0=S$ which is the affine coordinate ring of $X$. They are permuted by the Fourier transform, which fixes $P$ and $G_3$ and swaps the modules in each of the pairs
 \[(S,E),\ (G_{-1},D_1),\ (G_1,D_2),\ (G_2,G_4),\ (Q_1,Q_2),\ (D_0,Q_0),\]
 and are also permuted by the holonomic duality functor, which fixes $S,Q_0,P,D_0,E,G_3$ and swaps the modules in each of the pairs
 \[(D_1,D_2),\ (Q_1,Q_2),\ (G_2,G_4),\ (G_1,G_{-1}).\]
\end{thm-simples*}

We will analyze in more detail each of the $14$ simple modules in the theorem above, but before doing so we discuss the category $\opmod_{\GL}(\D_X)$ of $\GL$-equivariant coherent $\D_X$-modules. It is known quite generally that categories of $\D$-modules (or perverse sheaves) admit a quiver interpretation \cites{gel-mac-vil,vilonen,lor-wal}, but explicit descriptions of such categories are hard to come by \cites{galligo-granger-1,galligo-granger-2,mac-vil-cusp,braden-grinberg,lor-wal}. In the case of binary cubic forms we have the following.

\begin{thm-category*}
There is an equivalence of categories 
\[\opmod_{\GL}(\D_X) \simeq \rep(\Q,\I),\]
where $\rep(\Q,\I)$ is the category of finite-dimensional representations of a quiver $\Q$ with relations $\I$, described as follows. The vertices and arrows of the quiver $\Q$ are depicted in the diagram
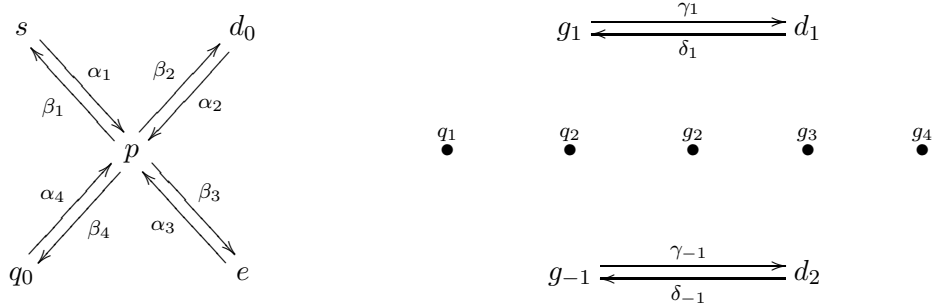
\[\xymatrix@=2.3pc@L=0.2pc{
s \ar@<0.5ex>[dr]^{\alpha_1} & & d_0 \ar@<0.5ex>[dl]^{\alpha_2}  & & & g_{1} \ar@<0.5ex>[rr]^{\gamma_1} & & d_1 \ar@<0.5ex>[ll]^{\delta_1} & \\
 & p \ar@<0.5ex>[ul]^{\beta_1} \ar@<0.5ex>[ur]^{\beta_2}\ar@<0.5ex>[dl]^{\beta_4} \ar@<0.5ex>[dr]^{\beta_3} & & & \overset{q_1}{\bullet} & \overset{q_2}{\bullet} & \overset{g_2}{\bullet} & \overset{g_3}{\bullet} & \overset{g_4}{\bullet} \\
q_0 \ar@<0.5ex>[ur]^{\alpha_4} & & e \ar@<0.5ex>[ul]^{\alpha_3} & & & g_{-1} \ar@<0.5ex>[rr]^{\gamma_{-1}} & & d_2 \ar@<0.5ex>[ll]^{\delta_{-1}} &
}\]
and the set of relations $\I$ is given by all 2-cycles and all non-diagonal compositions of two arrows:
\[\alpha_i \beta_i \mbox{ and } \beta_i\alpha_i \mbox{ for } i=1,2,3,4,\quad \gamma_i\delta_i\mbox{ and }\delta_i\gamma_i\mbox{ for }i=1,-1,\mbox{ and}\]
\[\alpha_1 \beta_2, \alpha_1\beta_4 , \alpha_2\beta_1, \alpha_2\beta_3, \alpha_3\beta_2,\alpha_3\beta_4,\alpha_4\beta_1,\alpha_4\beta_3.\]
Moreover, the bijection between the simple $\GL$-equivariant $\D_X$-modules and the nodes of the above quiver is given by replacing upper case symbols with the corresponding lower case symbols.
\end{thm-category*}

Once a quiver description of the category $\opmod_{\GL}(\D_X)$ is given, there are powerful tools available to understand its structure, such as Auslander--Reiten theory or tilting theory. We give an indication of this at the end of Section~\ref{sec:category-modDx} without going into the details of the theory. Instead we focus on proving that the quiver of $\opmod_{\GL}(\D_X)$ is of tame representation type, and to classify its indecomposable objects by relating them to the indecomposables of the extended Dynkin quiver of type $\hat{D}_4$. These latter indecomposables are well-understood, as their classification constitutes the solution of the celebrated \defi{four subspace problem} (see \cite{gel-pon} and \cite[Section XIII.3]{sim-sko2}).

In commutative algebra, the importance of the study of equivariant $\D$-modules comes from the desire to understand local cohomology modules with support in orbit closures \cites{raicu-weyman-witt,raicu-weyman,raicu-weyman-loccoh,raicu-veronese}. Our investigations owe a great deal to the work of Gennady Lyubeznik, who pioneered the use of $\D$-module techniques in commutative algebra. In the seminal paper \cite{lyubeznik}, Lyubeznik studies iterated local cohomology groups with respect to general families of supports, which lead to invariants known to be notoriously hard to compute in concrete examples. We show in Section~\ref{sec:loccoh} how basic knowledge of the structure of $\opmod_{\GL}(\D_X)$ can help to compute explicitly local cohomology groups associated with the space of binary cubics.

What is missing (and we promised to clarify) from the statements of the theorems above is the concrete construction of the simple equivariant $\D$-modules. In order to achieve this, we begin by setting up some notation. We let $V = W^*$ denote the dual vector space to $V$ so that $\Sym^3 V$ is naturally identified with the space of linear forms on $X$, and let $S = \Sym(\Sym^3 V)$ denote the ring of polynomial functions on $X$. If $\{w_0,w_1\}$ is a basis for $W$ then $\{w_0^3,3w_0^2 w_1,3w_0 w_1^2, w_1^3\}$ is a basis for $\Sym^3 W$. Choosing $\{x_0,x_1,x_2,x_3\}$ to be the dual basis of $\Sym^3 V$ we obtain an identification $S = \bb{C}[x_0,x_1,x_2,x_3]$. The condition for a cubic form 
\begin{equation}\label{eq:generic-cubic}
 f = x_0\cdot w_0^3 + 3x_1 \cdot w_0^2 w_1 + 3x_2 \cdot w_0 w_1^2 + x_3\cdot w_1^3
\end{equation}
to be in $\ol{O_2}$ is equivalent to the vanishing of the $2\times 2$ minors of 
\[ A=\begin{bmatrix}
 x_0 & x_1 & x_2 \\
 x_1 & x_2 & x_3 \\
\end{bmatrix}\]
while the condition that $f\in \ol{O_3}$ is equivalent to the vanishing of the discriminant
\begin{equation}\label{eq:discriminant}
 \Delta = 3 x_1^2 x_2^2 - 4x_0 x_2^3 - 4 x_1^3 x_3  - x_0^2 x_3^2 + 6 x_0x_1x_2x_3.
\end{equation}
We write $S_{\Delta}$ for the localization of $S$ at $\Delta$ and for every integer $i\in\bb{Z}$ we define
\begin{equation}\label{eq:Fi}
 F_i = S_{\Delta} \cdot \Delta^{i/6}
\end{equation} 
which is a $\D$-module. Moreover, we have $F_i = F_j$ if and only if $i-j$ is divisible by $6$. For $i=-1,0,1,2,3,4$ we define $G_i$ to be the $\D$-submodule of $F_i$ generated by $\Delta^{i/6}$, and part of the Theorem on Simple $\D$-modules is the fact that each $G_i$ is a simple $\D$-module. It is easy to see that $G_0=S$ is simple, and we will show later that in fact $G_i = F_i$ for $i=2,3,4$. Moreover, $F_1$ and $F_{-1}$ have $\D$-module length two so that the quotients $F_1/G_1$ and $F_{-1}/G_{-1}$ are simple $\D$-modules, which we label $D_1$ and $D_2$ respectively.

The modules $D_0$ and $E$ arise as local cohomology modules, for instance $D_0 = H^2_{\ol{O_2}}(S)$ and $E = H^4_{O_0}(S)$, while in order to construct $P$ we proceed as follows. The $\D$-module $F_0 = S_{\Delta}$ is generated by $\Delta^{-1}$, and therefore the same is true about the quotient $S_{\Delta}/S$. There exist a surjective $\D$-module map
\[ \pi: S_{\Delta}/S \lra E\]
sending $\Delta^{-1}$ to the socle generator of $E$, and $P$ can be realized as $\ker(\pi)$.

To construct the remaining modules $Q_0,Q_1,Q_2$ we use the Fourier transform, which may be defined as follows. If we write $\D = S \langle \pd_0,\pd_1,\pd_2,\pd_3\rangle$ with $\pd_i = \frac{\pd}{\pd x_i}$, then for every $\D$-module $M$ its \defi{Fourier transform $\mc{F}(M)$} is the $\D$-module with the same underlying abelian group, but with the \emph{new} action of the generators $x_i,\pd_i$ of $\D$ given in terms of the \emph{old} action by
\[ x_i \cdot^{new} m = - \pd_i \cdot^{old} m,\quad \pd_i \cdot^{new} m = x_i \cdot^{old} m\mbox{ for every }m\in M.\]
We use this definition only for the purpose of this Introduction, but in order to preserve equivariance more care is needed in defining the Fourier transform, as we explain in Section~\ref{subsec:Dmods}. The most basic example of Fourier transform is exhibited by the identification of $E$ with $\mc{F}(S)$, but this extends to other pairs of simple $\D$-modules as indicated in our first theorem. We have $Q_0 = \mc{F}(D_0)$ and
\[ Q_1 = (Q_0)_{\Delta} \oo_{S_{\Delta}} F_2 = (Q_0)_{\Delta} \cdot \Delta^{1/3},\quad Q_2 = (Q_0)_{\Delta} \oo_{S_{\Delta}} F_4 = (Q_0)_{\Delta} \cdot \Delta^{2/3},\]
finalizing the list of simple objects in $\opmod_{\GL}(\D_X)$.

The article is organized as follows. In Section~\ref{sec:prelim} we record some preliminaries and basic notation regarding $\GL$-representations, (equivariant) $\D$-modules, and quivers. In Section~\ref{sec:simple-Dmods} we analyze the simple $\GL$-equivariant $\D$-modules on the space of binary cubic forms, and describe their characters. In Section~\ref{sec:category-modDx} we study the category of equivariant $\D$-modules by providing its quiver description and classifying its indecomposable objects. We finish with a number of explicit computations of local cohomology modules in Section~\ref{sec:loccoh}.

\section{Preliminaries}\label{sec:prelim}

\subsection{Admissible representations and their characters}\label{subsec:repthy}

Throughout this paper we let $W$ denote a complex vector space of dimension two, we let $V = W^{\vee}$ be its dual, and write $\GL$ for the group $\GL(W)\simeq\GL_2(\bb{C})\simeq\GL(V)$ of invertible linear transformations of $W$. We let $\Lambda$ denote the set of (isomorphism classes of) finite dimensional irreducible $\GL$-representations, which are classified by \defi{dominant weights} $\ll=(\ll_1,\ll_2)\in \bb{Z}^2$, $\ll_1\geq\ll_2$. Concretely, we write $\S_{\ll}V$ for the irreducible $\GL$-representation corresponding to $\ll$, which is determined by the following conventions:
\begin{itemize}
 \item If $\ll=(0,0)$ then $\S_{\ll}V = \bb{C}$ with the trivial $\GL$-action.
 \item If $\ll=(a,0)$ with $a\geq 0$ then $\S_{\ll}V = \Sym^a V$ is the space of homogeneous degree $a$ polynomial functions on $W$.
 \item For every dominant $\ll=(\ll_1,\ll_2)$ we have $\S_{(\ll_1+1,\ll_2+1)}V = \S_{\ll}V \oo  \bw^2 V$, and in particular if $\ll=(1,1)$ then $\S_{(1,1)}V = \bw^2 V$. We will often refer to $\bw^2 V$ as the \defi{determinant of $V$} and denote it $\det(V)$.
\end{itemize}
 More generally, if $U$ is any $k$-dimensional vector space (or a $\GL$-representation), we define the \defi{determinant of $U$} to be $\det(U) = \bw^k U$, so for instance we get
 \begin{equation}\label{eq:det-sym}
 \det(\Sym^a V) = \bw^{a+1}\Sym^a V = \S_{(b,b)}V\mbox{ where }b=\frac{a(a+1)}{2}.
 \end{equation}
We write $\S_{\ll}W$ for the $\GL$-representation dual to $\bb{S}_{\ll}V$, so that if we let $\ll^{\vee} = (-\ll_2,-\ll_1)$ then
\[ \S_{\ll}W = \Hom_{\bb{C}}(\S_{\ll}V,\bb{C}) = \S_{\ll^{\vee}}V.\]
We next consider \defi{admissible $\GL$-representations} to be representations that decompose into a (possibly infinite) direct sum
\begin{equation}\label{eq:admissible}
 M = \bigoplus_{\ll} (\S_{\ll}V)^{\oplus a_{\ll}},\mbox{ where }a_{\ll}\geq 0.
\end{equation}
The \defi{Grothendieck group $\Gamma(\GL)$ of virtual admissible $\GL$-representations} is defined as a direct product of copies of $\bb{Z}$ indexed by the set $\Lambda$
\[ \Gamma(\GL) = \bb{Z}^{\Lambda} = \{\mbox{maps } f: \Lambda \lra \bb{Z} \}.\]
We will often represent elements of $\Gamma(\GL)$ as formal sums
\begin{equation}\label{eq:formalsum}
 a=\sum_{\ll} a_{\ll} \cdot e^{\ll}
\end{equation}
where $e^{\ll}$ corresponds to $\S_{\ll}V$ and $a_{\ll}\in\bb{Z}$ is the \defi{multipicity} of $\S_{\ll}V$. The correspondence between maps $f:\Lambda\lra\bb{Z}$ and the formal sum $a$ in (\ref{eq:formalsum}) is given by $a_{\ll} = f(\ll)$ for all $\ll\in\bb{Z}^2$ dominant. Given an element $a\in\Gamma(\GL)$ as in (\ref{eq:formalsum}) we will write
\[ \langle a,e^{\ll} \rangle = a_{\ll}.\]

Besides the additive group operation, $\Gamma(\GL)$ also has a partially defined multiplication given by convolution of functions:
\begin{equation}\label{eq:convolution}
 (f\cdot g)(\ll) = \sum_{\mu+\delta = \ll} f(\mu) \cdot g(\delta)
\end{equation}
which is defined precisely when all the sums in (\ref{eq:convolution}) involve finitely many non-zero terms. Using this multiplication, we can make sense of inverting some virtual representations such as $(1-e^{\ll})$:
\[ \frac{1}{1-e^{\ll}} = 1 + e^{\ll} + e^{2\ll} + \cdots\]
Given an admissible representation $M$ as in (\ref{eq:admissible}) we write $[M]$ for the class of $M$ in $\Gamma(\GL)$. For example, when $\Sym(V) = \bigoplus_{i\geq 0}\Sym^i V$ is the ring of polynomial functions on $W$ then $[\Sym(V)] = 1/(1-e^{(1,0)})$. The following will be important in our study of $\D$-modules on binary cubics:

\begin{lemma}[{\cite[Section~6.1]{landsberg-manivel}}]\label{lem:charS} If $S = \Sym(\Sym^3 V)$ is the ring of polynomial functions on the space $X=\Sym^3 W$ of binary cubic forms then
\begin{equation}\label{eq:character-S}
[S] = \frac{1+e^{(6,3)}}{(1-e^{(3,0)})(1-e^{(4,2)})(1-e^{(6,6)})}.
\end{equation}
\end{lemma}

We say that a sequence $(f_n)_{n}$ of elements in $\Gamma(\GL)$ \defi{converges to $f$}, and write $\lim_{n\to\infty} f_n = f$, if for each dominant weight $\ll$ we have that the sequence of integers $(f_n(\ll))_n$ is eventually constant, equal to $f(\ll)$. A typical example of convergent sequence arises from localization: if $\Delta$ is the discriminant of the binary cubic form (see (\ref{eq:discriminant})) then it spans a one-dimensional $\GL$-representation with $[\bb{C}\cdot\Delta] = e^{(6,6)}$. We get based on Lemma~\ref{lem:charS} that
\begin{equation}\label{eq:character-Sdelta}
 [S_{\Delta}] = \lim_{n\to\infty} [\Delta^{-n}\cdot S] = \frac{1+e^{(6,3)}}{(1-e^{(3,0)})(1-e^{(4,2)})} \cdot e^{(6,6)\bb{Z}}
\end{equation}
where $e^{(6,6)\bb{Z}} = \sum_{i\in\bb{Z}} e^{(6i,6i)}$. We have more generally the following.

\begin{lemma}\label{lem:localization}
 Suppose that $M$ is a $\GL$-equivariant $S$-module which is an admissible representation, and suppose that $\Delta$ is a non-zerodivisor on $M$. If for each dominant weight $\ll$ we have that the sequence of multiplicities $\langle [M],e^{\ll+(6n,6n)}\rangle$ is eventually constant, with stable value equal to $m_{\ll}^{stab}$, then
 \[ [M_{\Delta}] = \lim_{n\to\infty} [\Delta^{-n}\cdot M] = \sum_{\ll} m_{\ll}^{stab}\cdot e^{\ll}.\]
\end{lemma}

\begin{proof}
 The condition that $\Delta$ is a non-zerodivisor on $M$ insures that $M$ embeds into $M_{\Delta}$ and that one can write $M_{\Delta} = \bigcup_{n\geq 1} \Delta^{-n}\cdot M$. It then suffices to show that for each dominant weight $\ll$ we have that $\langle \Delta^{-n}\cdot M,e^{\ll} \rangle = m_{\ll}^{stab}$ for $n\gg 0$. Since $[\bb{C}\cdot\Delta] = e^{(6,6)}$, it follows that $[\bb{C}\cdot\Delta^n] = e^{(6n,6n)}$ and therefore
 \[\langle \Delta^{-n}\cdot M,e^{\ll} \rangle = \langle M,e^{\ll}\cdot[\Delta^n] \rangle = \langle [M],e^{\ll+(6n,6n)}\rangle = m_{\ll}^{stab} \mbox{ for }n\gg 0.\qedhere\]
\end{proof}

The main examples of admissible representations in this work come from $\GL$-equivariant coherent $\D_X$-modules (see \cite[Theorem~2.4]{lor-wal},
where admissible representations are called \defi{multiplicity-finite}). If $M$ is a $\GL$-equivariant $\D_X$-module then the same is true about the localization $M_{\Delta}$, hence Lemma~\ref{lem:localization} applies. In the next section we record some more generalities on equivariant $\D$-modules.

\subsection{Equivariant $\D$-modules {\cite{hot-tak-tan}}}\label{subsec:Dmods}

Given a smooth complex algebraic variety $Y$ we denote by $\D_Y$ the sheaf of differential operators on $Y$. If $G$ is a connected affine algebraic group acting on $Y$, with Lie algebra $\lie$, we get by differentiating the $G$-action on $Y$ a map from $\lie \oo_{\bb{C}} \mc{O}_Y$ to the sheaf of algebraic vector fields on $Y$, which in turn extends to a map of sheaves of algebras $U(\lie) \oo_{\bb{C}} \mc{O}_Y \to \D_Y$, where $U(\lie)$ denotes the universal enveloping algebra of $\lie$. We will always assume that $G$ acts on $Y$ with finitely many orbits and we will be interested in the study of
\[ \opmod_G(\D_Y) = \mbox{the category of }G\mbox{-equivariant coherent }\D_Y\mbox{-modules},\]
as defined in \cite[Definition 11.5.2]{hot-tak-tan}: a $\D_Y$-module $M$ is \defi{equivariant} if we have a $\D_{G\times Y}$-isomorphism $\tau: p^*M \rightarrow m^*M$, where $p: G\times Y\to Y$ denotes the projection and $m: G\times Y\to Y$ the map defining the action, with $\tau$ satisfying the usual compatibility conditions. When $Y$ is affine, this amounts to $M$ admitting an algebraic $G$-action whose differential coincides with the $\lie$-action induced by the natural map $\lie \to \D_Y$. It follows from \cite[Proposition~3.1.2]{VdB:loccoh} that a morphism of $\D_Y$-modules between objects in $\opmod_G(\D_Y)$ is automatically $G$-equivariant, so $\opmod_G(\D_Y)$ is equivalent to a full subcategory of the category $\opmod(\D_Y)$ of all coherent $\D_Y$-modules, which is closed under taking submodules and quotients. If $G$ is semi-simple, then if follows from \cite[Proposition~1.4]{lor-wal} that this subcategory is also closed under taking extensions. Moreover, it follows from \cite[Theorem~11.6.1]{hot-tak-tan} that every module in $\opmod_G(\D_Y)$ is regular and holonomic, and via the Riemann--Hilbert correspondence $\opmod_G(\D_Y)$ is equivalent to the category of $G$-equivariant perverse sheaves on $Y$. This category was shown in \cite[Theorem 4.3]{vilonen} to be equivalent to the category of finitely generated modules over a finite dimensional $\bb{C}$-algebra, which in turn is equivalent to the category of representations of a quiver with relations (see \cite[Corollary~I.6.10,~Theorems~II.3.7~and~III.1.6]{ass-sim-sko}). A more direct, $\D$-module theoretic approach to express $\opmod_G(\D_Y)$ as a category of quiver representations is described in \cite{lor-wal}, and it is this approach that we will follow in our study. 

In what follows it will be important to consider several other categories of $\D$-modules, as well as the functors between them. If $Z\subset Y$ is a $G$-stable closed subvariety of $Y$, we define $\opmod_G^Z(\D_Y)$ to be the full subcategory of $\opmod_G(\D_Y)$ consisting of equivariant $\D_Y$-modules with support contained in $Z$. Associated to any $G$-equivariant local system $\mc{S}$ on some open subset contained in the smooth locus of $Z$ we have (see \cite[Remark~7.2.10]{hot-tak-tan})
\[\mc{L}(Z,\mc{S};Y) = \mbox{the \defi{intersection homology} }\D_Y\mbox{-module corresponding to }\mc{S}\]
which is a simple object in $\opmod_G^Z(\D_Y)$ (and hence also simple in $\opmod_G(\D_Y)$ and in $\opmod(\D_Y)$). In the case when $\mc{S}$ is the constant sheaf $\bb{C}$, we simply write $\mc{L}(Z;Y)$ for the corresponding intersection homology $\D_Y$-module. One important construction of objects in $\opmod_G^Z(\D_Y)$ comes from considering local cohomology functors $\mc{H}^i_Z(Y,\bullet)$: for each $i\geq 0$ and each $M\in\opmod_G(\D_Y)$ we have
\[ \mc{H}^i_Z(Y,M) = \mbox{the }i\mbox{-th local cohomology module of }M\mbox{ with support in }Z,\]
which is an element of $\opmod_G^Z(\D_Y)$. If we write $c = \dim(Y) - \dim(Z)$ for the codimension of $Z$ in $Y$ and take $M=\mc{O}_Y$ to be the structure sheaf of $Y$, then
\begin{itemize}
 \item $\mc{H}^i_Z(Y,\mc{O}_Y) = 0$ for all $i<c$.
 \item $\mc{H}^c_Z(Y,\mc{O}_Y)$ contains $\mc{L}(Z;Y)$ as a $\D_Y$-submodule, and the quotient $\mc{H}^c_Z(Y,\mc{O}_Y) / \mc{L}(Z;Y)$ has support contained in the singular locus of $Z$ (in particular $\mc{H}^c_Z(Y,\mc{O}_Y) = \mc{L}(Z;Y)$ when $Z$ is smooth).
\end{itemize}
More generally, if $Z' \subset Z$ is closed, and if we let $V = Z \setminus Z'$ denote the complement, then we can talk about the local cohomology functors $\mc{H}^i_V(Y,\bullet)$ with support in $V$, which are related to the usual ones via a long exact sequence for every $\D_Y$-module $M$ (see \cite[(1.2)]{lyubeznik}):
\[ \cdots \lra \mc{H}^i_{Z'}(Y,M) \lra \mc{H}^i_Z(Y,M) \lra \mc{H}^i_V(Y,M) \lra \mc{H}^{i+1}_{Z'}(Y,M) \lra \cdots\]
Moreover, if $M\in\opmod_G(\D_Y)$ then we have that $\mc{H}^i_Z(Y,M)\in\opmod_G^Z(\D_Y)$.

When $Z$ is itself smooth we can talk about various categories of $\D_Z$-modules. As a consequence of Kashiwara's equivalence \cite[Theorem~1.6.1]{hot-tak-tan} we get an equivalence of categories
\begin{equation}\label{eq:Kashiwara-equiv}
 \opmod_G(\D_Z) \simeq \opmod^Z_G(\D_Y),
\end{equation}
under which $\mc{O}_Z \in \opmod_G(\D_Z)$ corresponds to $\mc{L}(Z;Y) = \mc{H}^c_Z(Y,\mc{O}_Y) \in \opmod^Z_G(\D_Y)$.

We next consider a non-empty $G$-stable open subset $U\subset Y$, and write $j:U \to Y$ for the corresponding open immersion. The direct and inverse image functors $j_*$ and $j^*$ for quasi-coherent sheaves restrict to functors between the corresponding categories of $\D$-modules:
\begin{equation}\label{eq:adjoint-pairs}
 \opmod(\D_U) \overset{j_*}{\underset{j^*}{\rightleftarrows}} \opmod(\D_Y), \qquad \opmod_G(\D_U) \overset{j_*}{\underset{j^*}{\rightleftarrows}} \opmod_G(\D_Y), \qquad \opmod_G^{W\cap U}(\D_U) \overset{j_*}{\underset{j^*}{\rightleftarrows}} \opmod_G^W(\D_Y),
\end{equation}
where in the last pair of categories $W$ denotes a closed $G$-stable subvariety of $Y$. In each of the three cases, we have that $j_*$ is right adjoint to $j^*$, and $j^*$ is (left) exact, so $j_*$ takes injective objects to injective objects. For any $\D_Y$-module $M$, the adjunction between $j_*$ and $j^*$ gives a natural map $M \to j_* j^*M$. If we let $Z = Y\setminus U$ we get an exact sequence
\begin{equation}\label{eq:4term-loccoh}
 0 \lra \mc{H}_Z^0(Y,M) \lra M \lra j_* j^*M \lra \mc{H}_Z^1(Y,M) \lra 0
\end{equation}
and for every $k\geq 1$ we have isomorphisms (where $R^k j_*$ denotes the $k$-th derived functor of $j_*$)
\[ R^k j_*(j^* M) \simeq \mc{H}_Z^{k+1}(Y,M).\]

As a consequence of the Riemann--Hilbert correspondence and \cite[Theorem~11.6.1]{hot-tak-tan} we have the following.
\begin{theorem}\label{thm:simples}
 Suppose that $Y$ is a smooth complex algebraic variety, and that $G$ is an algebraic group acting on $Y$.
 
 (a) If $G$ acts transitively on $Y$, we identify $Y$ with $G/K$ for some algebraic subgroup $K$ of $G$ and let $H = K/K^0$ denote the component group of $K$ (where $K^0$ is the connected component of the identity element in $K$). The category $\opmod_G(\D_Y)$ is equivalent to the category of finite dimensional complex representations of $H$ and is therefore semisimple (since $H$ is finite).

 (b) If $G$ acts with finitely many orbits on $Y$ then we have a bijective correspondence
 \[\left\{(O,\mc{S}) : \begin{array}{c}
 O\mbox{ is a }G\mbox{-orbit and } \\ \mc{S}\mbox{ is an equivariant irreducible local system on }O
 \end{array}
 \right\} 
 \llra \{\mbox{simple objects of }\opmod_G(\D_Y)\}\]
 \[\mbox{given by }\qquad(O,\mc{S}) \llra \mc{L}(O,\mc{S};Y).\]
 Moreover, if we fix $O = G/K$ and let $H = K/K^0$ as in part (a), then we have a bijective correspondence
\[\{\mbox{equivariant irreducible local systems on }O\} \llra \{\mbox{irreducible representations of }H\}.\]
\end{theorem}

The following will play a key role in our description of the category of equivariant $\D$-modules in Section~\ref{sec:category-modDx}.

\begin{lemma}\label{lem:inj}
Let $G,Y$ be as in Theorem~\ref{thm:simples}, and let $O$ denote one of the orbits. Consider $Z=\ol{O}\setminus O$, $U=Y\setminus Z$, and let $j: U \to Y$ be the corresponding open immersion. If $M$ is a simple equivariant $\D_Y$-module with support $\ol{O}$ then $j_* j^* M$ is the injective envelope of $M$ in the category $\opmod^{\ol{O}}_G (\D_Y)$.
\end{lemma}

\begin{proof}
Since $M$ is simple with support $\ol{O}$ and $\mc{H}_Z^0(Y,M)$ is a $\D$-submodule with support contained in $Z\subsetneq\ol{O}$, it follows that $\mc{H}_Z^0(Y,M)=0$. Using (\ref{eq:4term-loccoh}) we get a natural injective map $M\hra j_* j^* M$, so it suffices to verify that $j_*j^*M$ is an injective object in $\opmod^{\ol{O}}_G (\D_Y)$ and that it is indecomposable.

If we let $W = \ol{O}$ then $W\cap U = O$, so based on (\ref{eq:adjoint-pairs}) we can think of $j_*,j^*$ as a pair of adjoint functors between the categories $\opmod^{{O}}_G (\D_U)$ and $\opmod^{\ol{O}}_G (\D_Y)$. Using Kashiwara's equivalence (\ref{eq:Kashiwara-equiv}) and Theorem~\ref{thm:simples}(a) we get that $\opmod^{O}_G (\D_U)\simeq \opmod_G (\D_O)$ is semisimple and in particular $j^*M$ is an injective object. As remarked earlier, $j_*$ takes injectives to injectives, hence $j_*j^*M$ is injective in $\opmod^{\ol{O}}_G (\D_Y)$.

To show that  $j_* j^* M$ is indecomposable, we note that  $j^*M$ is simple and therefore
\[\Hom_{\D_Y} (j_* j^* M, j_* j^* M) = \Hom_{\D_U} (j^*j_* j^* M, j^* M)=\Hom_{\D_U} (j^* M, j^* M)=\bb{C}.\]
If $j_* j^* M$ decomposed as a direct sum of two submodules then $\Hom_{\D_Y} (j_* j^* M, j_* j^* M)$ would have dimension at least two, concluding our proof.
\end{proof}

We end this section by recalling two important functors on $\opmod_G(\D_Y)$:
\begin{itemize}
 \item The \defi{holonomic duality functor} $\bb{D}: \opmod_G(\D_Y) \lra \opmod_G(\D_Y)$ is a \defi{duality of categories} (i.e. an equivalence of categories between $\opmod_G(\D_Y)$ and its opposite category $\opmod_G(\D_Y)^{op}$). Translated via the Riemann--Hilbert correspondence this functor corresponds to Verdier duality \cite[Corollary~4.6.5]{hot-tak-tan}, and in particular it restricts to a duality on $\opmod_G^Z(\D_Y)$ for any $G$-stable subvariety $Z$. At the level of $\D_Y$-modules it is defined via (see \cite[Section~2.6]{hot-tak-tan})
 \[ \bb{D}(M) = \ShExt^{\ n}_{\D_Y}(M,\D_Y) \oo_{\mc{O}_Y} \omega_Y^{-1}\]
where $n=\dim(Y)$, $\omega_Y$ is the canonical line bundle on $Y$, and where $\bullet \oo_{\mc{O}_Y} \omega_Y^{-1}$ is the functor that transforms right $\D_Y$-modules (of which $\ShExt^{\ n}_{\D_Y}(M,\D_Y)$ is one) into left $\D_Y$-modules \cite[Proposition~1.2.12]{hot-tak-tan}. What will be important for us is that $\bb{D}$ interchanges injective and projective objects in $\opmod_G^Z(\D_Y)$, and that it permutes the collection of simple objects. 
 
 \item The \defi{Fourier transform}. If $Y$ is an affine space, corresponding to a finite dimensional $G$-representation $U$, we let $Y^{\vee}$ denote the affine space corresponding to the dual representation $U^{\vee}$. As explained in \cite[Section~2.5]{raicu-matrices} we get an equivalence of categories
 \[\mc{F}^{\circ}:\opmod_G(\D_Y) \lra \opmod_G(\D_{Y^{\vee}}),\quad \mc{F}^{\circ}(M) = M \oo_{\bb{C}} \det(U^{\vee}).\] 
In the case when $G = \GL(W)$ and $Y=\S_{\ll}W$, we can further use the natural isomorphism $\tau:\GL(W) \simeq \GL(V)$ given by $\tau(\phi) = (\phi^{\vee})^{-1}$ (the inverse transpose), which is compatible with the respective actions of $\GL(V)$ and $\GL(W)$ on $\S_{\ll}V$, to obtain an equivalence of categories
\[ \opmod_{\GL(W)}(\D_{\S_{\ll}V}) \overset{\tau}{\lra} \opmod_{\GL(V)}(\D_{\S_{\ll}V})\]
Fixing a vector space isomorphism $T:V\to W$ we obtain a group isomorphism $\tau':\GL(V) \simeq \GL(W)$ given by $\tau'(\phi) = T\circ\phi\circ T^{-1}$, and because of the functoriality of $\bb{S}_{\ll}$ an isomorphism $\tl{T}:\S_{\ll}V \lra S_{\ll}W$ compatible with the respective actions of $\GL(V)$ and $\GL(W)$ on the source and the target. This provides a further equivalence of categories
\[ \opmod_{\GL(V)}(\D_{\S_{\ll}V}) \overset{(\tau',\tl{T})}{\lra} \opmod_{\GL(W)}(\D_{\S_{\ll}W}).\]
Putting all of these equivalences together, we obtain a self equivalence of categories
\begin{equation}\label{eq:def-Fourier}
\mc{F}:\opmod_{\GL(W)}(\D_{\S_{\ll}W}) \overset{\mc{F}^{\circ}}{\lra} \opmod_{\GL(W)}(\D_{\S_{\ll}V}) \overset{\tau}{\lra} \opmod_{\GL(V)}(\D_{\S_{\ll}V}) \overset{(\tau',\tl{T})}{\lra} \opmod_{\GL(W)}(\D_{\S_{\ll}W})
\end{equation}
which we will refer to (by abuse of language) as the Fourier transform on $Y=\S_{\ll}W$.
\end{itemize}

\begin{remark}\label{rem:Fourier} The construction of a Fourier transform as a self-equivalence of $\opmod_G(\D_{Y})$ can be done more generally when $G$ is a connected linear reductive group. Let $T\subset G$ be a maximal torus and let $B\subset G$ be a Borel subgroup containing $T$. There exists an involution $\theta \in \Aut(G)$ such that $\theta(t)=t^{-1}$ for all $t\in T$ and $B\cap \theta(B)=T$ (for example, see \cite[II. Corollary 1.16]{jantzen}). If $V$ is any $G$-module, we can twist the action of $G$ by $\theta$ to obtain a $G$-module $V^*$, which is isomorphic to the usual dual representation $V^{\vee}$ of $V$. Twisting the action of $G$ on $Y$ by $\theta$ gives then an equivalence of categories $\tau_{\theta}:\opMod_G(\D_{Y}) \simeq \opMod_G(\D_{Y^*})$. Using a $G$-isomorphism $Y^*\simeq Y^{\vee}$ together with the usual Fourier transform $\mathcal{F}^{\circ} : \opMod_G(\D_{Y^{\vee}}) \xrightarrow{\sim} \opMod_G(\D_{Y})$ we obtain an involutive self-equivalence of $\opmod_G(\D_Y)$ (which we also call the Fourier transform)
\[\mathcal{F} : \opmod_G(\D_Y) \overset{\tau_{\theta}}{\lra} \opmod_G(\D_{Y^*}) \xrightarrow{\sim} \opmod_G(\D_{Y^{\vee}}) \overset{\mc{F}^{\circ}}{\lra} \opMod_G(\D_{Y}).\]
\end{remark}

\subsection{The Fourier transform of admissible representations}\label{subsec:Fourier-admissible}

We let $U = \Sym^3 W$ and recall from (\ref{eq:det-sym}) that
\[ \det(U) = \S_{(6,6)}W.\]
We define the Fourier transform (relative to $U$) $\mc{F}:\Lambda\to\Lambda$ via
\[ \mc{F}(\ll) = \ll^{\vee} - (6,6) = (-\ll_2-6,-\ll_1-6).\]
This in turn induces a Fourier transform $\mc{F}:\Gamma(\GL) \lra \Gamma(\GL)$ given by
\[\mc{F}\left(\sum_{\ll} a_{\ll} e^{\ll}\right) = \sum_{\ll} a_{\ll} e^{\mc{F}(\ll)}.\]
The motivation behind this definition is as follows. If $X = \Sym^3 W$ is the affine space corresponding to $U$, and if $M\in\opmod_{\GL}(\D_X)$ then it follows from \cite[Theorem~2.4]{lor-wal} that $[M]$ is an admissible $\GL$-representation, and therefore
\[ M \simeq \bigoplus_{\ll} \S_{\ll}V^{\oplus a_{\ll}},\mbox{ for }a_{\ll}\in\bb{Z}.\]
 Applying the Fourier transform $\mc{F}:\opmod_{\GL}(\D_X)\lra \opmod_{\GL}(\D_X)$ as constructed in (\ref{eq:def-Fourier}) we obtain
 \[\mc{F}(M) \simeq \bigoplus_{\ll} \S_{\ll+(6,6)}W^{\oplus a_{\ll}} = \bigoplus_{\ll} \S_{\ll^{\vee}-(6,6)}V^{\oplus a_{\ll}},\mbox{ that is }[\mc{F}(M)]=\mc{F}([M]).\]
We rewrite the above conclusion as follows:
\begin{equation}\label{eq:wts-Fourier}
 \langle [\mc{F}(M)], e^{\ll} \rangle = \langle [M], e^{\ll^{\vee} - (6,6)} \rangle \mbox{ for every dominant weight }\ll.
\end{equation}

We record two basic instances of the Fourier transform which will be used later. Using the fact that the simple $\D$-module $E = \mc{L}(O_0;X)$ supported at the origin can be realized as $\mc{F}(S)$ it follows from (\ref{eq:character-S}) that
\begin{equation}\label{eq:character-E}
[E] = [\mc{F}(S)] = \frac{1+e^{(-3,-6)}}{(1-e^{(0,-3)})(1-e^{(-2,-4)})(1-e^{(-6,-6)})}\cdot e^{(-6,-6)}.
\end{equation}
We can also use (\ref{eq:wts-Fourier}) in conjunction with (\ref{eq:character-Sdelta}) to show that the character of $S_{\Delta}$ doesn't change under taking Fourier transform:
\[ [\mc{F}(S_{\Delta})] = \frac{1+e^{(-3,-6)}}{(1-e^{(0,-3)})(1-e^{(-2,-4)})} \cdot e^{(6,6)\bb{Z}} \cdot e^{(-6,-6)} = [S_{\Delta}].\]
This is a reflection of the fact that $\mc{F}(S_{\Delta})$ has three composition factors, $S,E$ and $P$, and the Fourier transform exchanges $S$ and $E$, and it preserves $P$.

\subsection{Quivers and their representations}\label{subsec:quivers}
We begin by establishing some notation and reviewing some basic results regarding the representation theory of quivers, following \cite{ass-sim-sko}.  A \defi{quiver} $\Q$ is an oriented graph, i.e. a pair $\Q=(\Q_0,\Q_1)$ formed by a finite set of vertices $\Q_0$ and a finite set of arrows $\Q_1$. An arrow $\a\in \Q_1$ has a \defi{head} (or a \defi{target}) $h(\a)$ and a \defi{tail} (or a \defi{source}) $t(\a)$ which are elements of $\Q_0$:
\[\xymatrix{
t(\a) \ar[r]^{\a} & h(\a)
}\]
A \defi{(directed) path $p$ in $\Q$ from $a$ to $b$ of length $l$} is a sequence of arrows
\begin{equation}\label{eq:pathQ}
p:\quad\xymatrix{
a=a_0 \ar[r]^{\a_1} & a_1 \ar[r]^{\a_2} & a_2 \ar[r]^{\a_3} & \ \cdots\ \ar[r]^{\a_l} & a_l = b \\
}
\end{equation}
where $a_i\in \Q_0$ and $\a_i\in \Q_1$. We call $a$ (resp. $b$) the \defi{source} (resp. \defi{target}) of the path, and write $p=\a_1\a_2\cdots\a_l$. The complex vector space with basis given by the paths in $\Q$ has a natural multiplication induced by concatenation of paths (where $pq=0$ if the source of $q$ is different from the target of $p$). The corresponding $\bb{C}$-algebra is called the \defi{path algebra of the quiver $\Q$} and is denoted $\bb{C}\Q$. A \defi{relation} in $\Q$ is a $\bb{C}$-linear combination of paths of length at least two having the same source and target. We define a \defi{quiver with relations} $(\Q,\mc{I})$ to be a quiver $\Q$ together with a finite set of relations $\mc{I}$. The \defi{quiver algebra of $(\Q,\mc{I})$} is the quotient $\mc{A}=\bb{C}\Q/\langle\mc{I}\rangle$ of the path algebra by the ideal generated by the relations. We will often use the word quiver to refer to a quiver with relations $(\Q,\I)$, and talk about $\sum c_i\cdot p_i = 0$ as being a relation in the quiver if $\sum c_i\cdot p_i \in \langle\mc{I}\rangle$, or equivalently if the relation $\sum c_i\cdot p_i = 0$ holds in the quiver algebra $\mc{A}$. We will moreover always assume that the ideal of relations $\langle\mc{I}\rangle$ contains any path whose length is large enough, so that the corresponding quiver algebra $\mc{A}$ is finite dimensional (see \cite[Section II.2]{ass-sim-sko}).

A \defi{(finite-dimensional) representation} $V$ of a quiver $(\Q,\mc{I})$ is a family of (finite-dimensional) vector spaces $\{V_x\,|\, x\in Q_0\}$ together with linear maps $\{V(\a) : V_{t(\a)}\to V_{h(\a)}\, | \, \a\in Q_1\}$ satisfying the relations induced by the elements of $\mc{I}$. More precisely, for every path $p:a\to b$ as in (\ref{eq:pathQ}) we consider the composition
\[ V(p) = V(\a_l) \circ V(\a_{l-1}) \circ \cdots \circ V(\a_1)\]
and we ask that for every element $\sum_i c_i\cdot p_i\in\mc{I}$ with $c_i\in\bb{C}$ and $p_i$ a path from $a$ to $b$, we have that
\[ \sum_i c_i\cdot V(p_i) = 0.\]
A morphism $\phi:V\to V'$ of two representations $V,V'$ of $(\Q,\mc{I})$ is a collection of linear maps 
\[\phi = \{\phi(x) : V_x \to V'_x\,| \,x\in Q_0\},\]
with the property that for each $\a\in Q_1$ we have 
\[\phi(h(\a))\circ V(\a)=V'(\a)\circ \phi(t(\a)).\]
We note that the data of a representation of $(\Q,\mc{I})$ is equivalent to that of a module over the quiver algebra $\mc{A}$, and a morphism of representations corresponds to an $\mc{A}$-module homomorphism. In other words, the category $\rep(\Q,\mc{I})$ of finite-dimensional representations of $(\Q,\mc{I})$ is equivalent to that of finitely generated $\mc{A}$-modules \cite[Section~III.1, Thm.~1.6]{ass-sim-sko}. Moreover, this is an abelian category with enough projectives and injectives, and having finitely many simple objects, as we explain next.

The (isomorphism classes of) simple objects in $\rep(\Q,\mc{I})$ are in bijection with the vertices of $\Q$. For each $x\in \Q_0$, the corresponding simple $\sS^x$ is the representation with 
\begin{equation}\label{eq:defS}
(\sS^x)_x=\bb{C},\ (\sS^x)_y=0\mbox{ for all }y\in \Q_0\setminus\{x\},\mbox{ and }\sS^x(\a)=0\mbox{ for all }\a\in\Q_1.
\end{equation}
A representation of $(\Q,\mc{I})$ is called \defi{indecomposable} if it is not isomorphic to a direct sum of two non-zero representations. Just like the simple objects, the indecomposable projectives (resp. injectives) in $\rep(\Q,\mc{I})$ are in bijection with the vertices of $\Q$. For each $x\in \Q_0$, we let $\sP^x$ (resp. $\sI^x$) denote the \defi{projective cover} (resp. \defi{injective envelope}) of~$\sS^x$, as constructed in \cite[Section III.2]{ass-sim-sko}. For $y\in\Q_0$, the dimension of $(\sP^x)_y$ (resp. $(\sI^x)_y$) is given by the number of paths from $x$ to $y$ (resp. from $y$ to $x$), considered up to the relations in $\mc{I}$. More precisely
\begin{equation}\label{eq:defP-I}
 (\sP^x)_y = \mbox{Span}\{ p\in\mc{A}:p\mbox{ is a path from }x\mbox{ to }y\},\ (\sI^x)_y = \mbox{Span}\{ p\in\mc{A}:p\mbox{ is a path from }y\mbox{ to }x\}^{\vee},
\end{equation}
where $^{\vee}$ denotes as usual the dual vector space. For an arrow $\a\in\Q_0$, thought of as an element of $\mc{A}$, we have that $\sP^x(\a)$ is right multiplication by $\a$, while $\sI^x(\a)$ is the dual to left multiplication by $\a$.

Unlike the classification of simple objects, and that of indecomposable injective and projective objects, the classification of general indecomposable objects in $\rep(\Q,\mc{I})$ is significantly more involved. A quiver $(\Q,\mc{I})$ is said to be of \defi{finite representation type} if it has finitely many (isomorphism types of) indecomposable representations. It is of \defi{tame representation type} if all but a finite number of indecomposable representations of $(\Q,\mc{I})$ of a given dimension belong to finitely many one-parameter families \cite[XIX.3, Definition 3.3]{sim-sko3}, and it is of \defi{wild representation type} otherwise. As the name suggests, in the wild case the classification of indecomposables is essentially intractable, while in the finite and tame cases it is more manageable. A special instance of tame representation type, and the one that will concern us in this paper, is when the number of one-parameter families of indecomposables is bounded as the dimension of the representations grows. We say in this case that the quiver $(\Q,\mc{I})$ is of \defi{domestic tame representation type} \cite[XIX.3, Definitions 3.6, 3.10 and Theorem 3.12]{sim-sko3}. One example of such a quiver, which will appear again in Section~\ref{subsec:indecomp}, is given by the following.

\begin{example}\label{ex:dee4hat}
Let $Q$ be the extended Dynkin quiver $\hat{D}_4$ (with no relations):
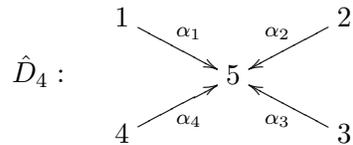
\[\hspace{-0.2in} \hat{D}_4: \hspace{0.15in} \begin{aligned} \xymatrix@R=0.7pc@C=2.5pc{
1 \ar[dr]^{\alpha_1}  & & 2 \ar[dl]_{\alpha_2}\\
 & 5 & \\
4 \ar[ur]_{\alpha_4} & & 3\ar[ul]^{\alpha_3}
}\end{aligned}\]
The representation theory of the the quiver $\hat{D}_4$ is well-understood. The quiver is of (domestic) tame representation type, and all its indecomposable representations have been classified -- this is, in the language of quivers, the famous four subspace problem solved in \cite{gel-pon}. For the complete description of the indecomposables of $\hat{D}_4$ together with the Auslander-Reiten quiver, we refer the reader to \cite[Section XIII.3]{sim-sko2}. In the following we illustrate only the 1-parameter families of indecomposable representations (they are taken from \cite[XIII.3. Table 3.14]{sim-sko2}): for each $n>0$ there exists precisely one $1$-parameter family of indecomposable representations $V$ of $\hat{D}_4$ with dimension vector $(n,n,n,n,2n)$ (i.e. with $\dim V_1 = \dim V_2 = \dim V_3 = \dim V_4 = 1$ and $\dim V_5 = 2n$), and moreover these are all the $1$-parameter families occuring in the classification of indecomposable of $\hat{D}_4$. 

Let $I_n$ denote the $n\times n$ identity matrix, and for any $\lambda\in\bb{C}$ we denote by $J_n(\lambda)$ the $n\times n$ Jordan block
\[J_n(\lambda) = \begin{bmatrix}
\lambda & 1 & 0 & \dots & 0 \\
0 & \lambda & 1 & \dots & 0 \\
\vdots & \vdots & \ddots &\ddots & \vdots\\
0 & 0 & \dots & \lambda & 1 \\
0 & 0 & \dots & 0 & \lambda
\end{bmatrix}.\]
Then for any $n>0$, we have the following 1-parameter family of indecomposables $R_n(\lambda)$, where $\lambda \in \bb{C}$:
\[\hspace{-0.2in} R_n(\lambda) : \hspace{0.15in} \begin{aligned} \xymatrix@R=2pc@C=4pc{
\bb{C}^n \ar[dr]^{\footnotesize\begin{bmatrix} I_n\\ 0 \end{bmatrix}}  & & \bb{C}^n  \ar[dl]_{\footnotesize\begin{bmatrix} 0\\I_n \end{bmatrix}}\\
 & \bb{C}^{2n} & \\
\bb{C}^n \ar[ur]_{\footnotesize\begin{bmatrix} I_n\\ J_n(\lambda) \end{bmatrix}} & & \bb{C}^n\ar[ul]^{\footnotesize\begin{bmatrix} I_n\\ I_n \end{bmatrix}}
}\end{aligned}\]
\end{example}

In Section~\ref{sec:category-modDx} we will be concerned with the quiver associated with the category of $\GL$-equivariant $\D$-modules on the space of binary cubic forms. We will prove that the said quiver is of (domestic) tame representation type and we will classify its indecomposables by directly relating them to indecomposables of~$\hat{D}_4$. The following reduction lemmas will play a key role in describing this relationship.

Consider a vertex $x$ in a quiver $(\Q,\mc{I})$:
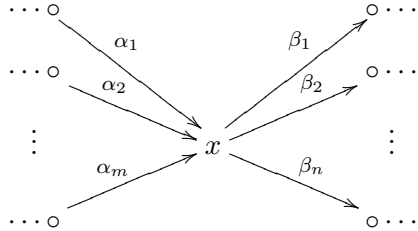
\[\xymatrix@C=4pc@R=1pc{
	\cdots \circ \ar@<0.6ex>[ddr]^(0.4){\a_1} & & \circ \cdots\\
\cdots  \circ \ar[dr]^(0.4){\a_2} & & \circ \cdots\\
\vdots  & x  \ar@<0.6ex>[uur]^(0.6){\b_1} \ar[ur]^(0.6){\b_2}  \ar[dr]^{\b_n} & \vdots \\
\cdots  \circ \ar[ur]^{\a_m} & & \circ  \cdots
}\]
Following \cite{martinez}, we say $x$ is a \defi{node} of $(\Q,\mc{I})$ if we have $\a_i\b_j = 0$ for all  $1\leq i \leq m, 1\leq j \leq n$. If $x$ is a node of $(\Q,\mc{I})$, we can ``separate the node" to obtain a new quiver $(\Q', \I')$:

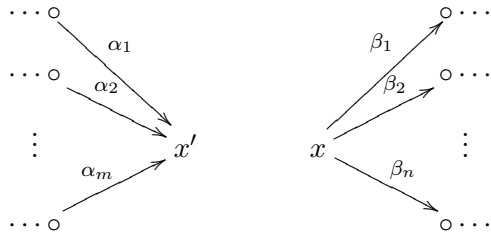
\[\xymatrix@C=3pc@R=1pc{
	\cdots \circ \ar@<0.7ex>[ddr]^(0.4){\a_1} & & & \circ \cdots\\
\cdots  \circ \ar@<0.2ex>[dr]^(0.4){\a_2} & & & \circ \cdots\\
\vdots  & x' & x  \ar@<0.7ex>[uur]^(0.6){\b_1} \ar@<0.2ex>[ur]^(0.6){\b_2}  \ar[dr]^{\b_n} & \vdots \\
\cdots  \circ \ar[ur]^{\a_m} & & & \circ  \cdots
}\]
The relations in $\I'$ are the natural ones obtained from $\I$ by removing the relations $\a_i \b_j = 0$. If $V$ is a representation of $(\Q,\I)$, then we can induce a representation $V'$ of $(\Q',\I')$ by letting 
\[V'_{x'} = \sum_{i=1}^{m} \operatorname{Im} V(\a_i) \subset V_x \mbox{ and } V'_{x} = V_x / V'_{x'}.\]
The following is a consequence of \cite[Theorem 2.10(b)]{martinez} (albeit formulated in a slightly different language):
\begin{lemma}\label{lem:node}
The functor $V\mapsto V'$ induces a bijection between the (isomorphism classes of) non-simple indecomposable objects in $\rep(\Q,\I)$ and those in $\rep(\Q',\I')$.
\end{lemma}

\vspace{0.1in}

Consider next the following situation. Suppose that $(\Q,\I)$ is the quiver
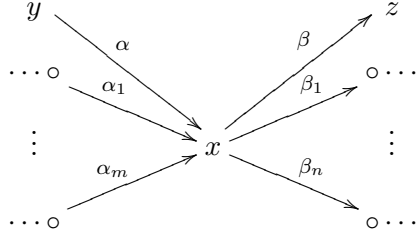
\[\xymatrix@C=4pc@R=1pc{
	y \ar@<0.6ex>[ddr]^(0.4){\a} & & z\\
\cdots  \circ \ar[dr]^(0.4){\a_1} & & \circ \cdots\\
\vdots  & x  \ar@<0.6ex>[uur]^(0.6){\b} \ar[ur]^(0.6){\b_1}  \ar[dr]^{\b_n} & \vdots \\
\cdots  \circ \ar[ur]^{\a_m} & & \circ  \cdots
}\]
with relations $\a_i\cdot\b= 0$ and $\a\cdot\b_j=0$, for all $1\leq i \leq m,1\leq j \leq n$, where the diagram above is indicative of the fact that the only arrow connected to the vertex $y$ (resp. $z$) is $\a$ (resp. $\b$). Let $\sP^y$ be the indecomposable projective corresponding to $y$, and let $\sI^z$ be the indecomposable injective corresponding to $z$. It follows from the explicit description in (\ref{eq:defP-I}) that $\sP^y\simeq\sI^z$ with
\[(\sP^y)_y = (\sP^y)_x = (\sP^y)_z = \bb{C},\ \sP^y(\a) = \sP^y(\b) = \mbox{id}_{\bb{C}},\mbox{ and }(\sP^y)_t = 0\mbox{ for }t\neq x,y,z.\]
We have moreover the following.

\begin{lemma}\label{lem:addrel}
If $(\Q,\I)$ is a quiver as above and if $V$ is an indecomposable representation of $(\Q,\I)$ with $V\not\simeq \sP^y$, then $V(\b) \circ V(\a) = 0$.
\end{lemma}

\begin{proof}
We let $A= \sum_{j=1}^m \operatorname{im} V(\a_j)$ and denote $i: A\hookrightarrow V_x$ the inclusion. We consider the Dynkin quiver of type $D_4$ (with no relations)
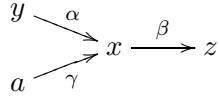
\[
\xymatrix@R=0.2pc{ 
y \ar[dr]^{\a} &  \\
 & x \ar[r]^{\b} & z \\
 a  \ar[ur]_{\gamma} & & 
}
\]
and its representation $V'$ given by
\[V': \hspace{0.2in} \begin{aligned}\xymatrix@R=0.2pc{ 
V_y \ar[dr]^{V(\a)} &  \\
 & V_x \ar[r]^{V(\b)} & V_z \\
 A   \ar[ur]_i & & 
}\end{aligned}\]
If we assume that $V(\b) \circ V(\a) \neq 0$, then $V'$ has an indecomposable summand $X$ with the property that $X(\b) \circ X(\a) \neq 0$. Moreover, it follows from the relations in $(\Q,\I)$ that $V'(\b) \circ V'(\gamma) = V(\b) \circ i = 0$, hence $X(\b) \circ X(\gamma) = 0$. The classification in \cite[Chapter VII,~Example~5.15(b)]{ass-sim-sko} of indecomposables of the Dynkin quiver of type $D_4$ shows that there is only one indecomposable $X$ with $X(\b) \circ X(\a) \neq 0$ and $X(\b) \circ X(\gamma) = 0$, namely
\[X: \hspace{0.2in} \begin{aligned}\xymatrix@R=0pc{ 
\bb{C} \ar[dr]^{1} & &  \\
 & \bb{C} \ar[r]^{1} & \bb{C} \\
 0  \ar[ur] & & 
}\end{aligned}\]
We can then write $V'\simeq X^{\oplus k} \oplus Y$, for some $k\geq 1$, where $Y$ is a representation of the quiver of type $D_4$ with $Y(\b)\circ Y(\a)=0$. Due to the relations $\a \cdot \b_j=0$, we can lift this isomorphism to a decomposition $V\simeq(\sP^y)^{\oplus k} \oplus Z$ in $\rep(\Q,\I)$, where $Z$ is a representation of $(\Q,\I)$ with $Z(\a) = Y(\a)$, $Z(\b)=Y(\b)$. 
 Since $V$ was indecomposable, we conclude that $k=1$ and $Z=0$, i.e. $V\simeq\sP^y$ which is a contradiction. It follows that we must have $V(\b) \circ V(\a)=0$, concluding the proof.
\end{proof}

\section{The simple equivariant $\D$-modules on binary cubics}\label{sec:simple-Dmods}

In this section we give a classification of the simple $\GL$-equivariant holonomic $\D$-modules on the space $X = \Sym^3 W$ of binary cubic forms. For each of these simple modules we give an explicit description of the characters of the underlying admissible representations. 

\subsection{The classification of simple equivariant $\D$-modules}\label{subsec:classification}

We will use Theorem~\ref{thm:simples} in order to obtain a classification of the simple equivariant $\D$-modules. We choose a basis $\{w_0,w_1\}$ for $W$ and identify $\GL\simeq\GL_2(\bb{C})$ relative to this basis. The non-zero elements of $X = \Sym^3 W$ are homogeneous cubic polynomials in the variables $w_0,w_1$ as in (\ref{eq:generic-cubic}), and as such they factor into a product of three linear forms. The $\GL$-orbit structure of $X$ is then described by the different types of factorizations: a non-zero cubic form may have three distinct (up to scaling) linear factors, or precisely one factor that is repeated twice, or it could be the cube of a linear form. The following table records representatives of each orbit $O_i$ ($i=0,2,3,4$), together with their stabilizers $K_i$ and the component groups $H_i = K_i/K_i^0$ of these stabilizers (we write $1$ for the trivial group, $C_k$ for the cyclic group of order $k$, and $\rtimes$ for a semidirect product).

\begin{center}
\renewcommand{\arraystretch}{1.5}
\begin{tabular}{c|c|c|c}
Orbit & Representative & Stabilizer & Component group \\
\hline
$O_0$ & $0$ & $\GL$ & 1 \\
\hline & & &\\[-10pt]
$O_2$ & $w_0^3$ & $\left\{\begin{pmatrix} \xi & \b \\ 0 & \a\end{pmatrix}:\xi^3 = 1,\a\in\bb{C}^{\times},\b\in\bb{C}\right\}$ & $C_3$ \\[15pt]
\hline & & &\\
$O_3$ & $w_0^2\cdot w_1$ & $\left\{\begin{pmatrix} \a & 0 \\ 0 & \a^{-2} \end{pmatrix}:\a\in\bb{C}^{\times}\right\}$ & 1 \\[15pt]
\hline & & &\\
$O_4$ & $w_0^3 + w_1^3$ & $\left\{\begin{pmatrix} \xi_1 & 0 \\ 0 & \xi_2 \end{pmatrix}:\xi_1^3 = \xi_2^3 = 1\right\} \bigcup \left\{\begin{pmatrix} 0 & \xi_1 \\  \xi_2 & 0 \end{pmatrix}:\xi_1^3 = \xi_2^3 = 1\right\}$ & $(C_3\times C_3)\rtimes C_2$ \\[10pt]
& & & \\
\end{tabular}
\end{center}
We note that the semidirect product $H_4 = K_4 = (C_3\times C_3)\rtimes C_2$ (where $C_2$ acts on $C_3\times C_3$ by interchanging the factors) is in fact isomorphic to $C_3 \times S_3$ (where $S_3$ denotes the group of permutations on $3$ letters) if we make the identifications
\[ C_3 \simeq \left\{\begin{pmatrix} \xi & 0 \\ 0 & \xi \end{pmatrix}:\xi^3 = 1\right\} \mbox{ and }S_3\simeq\left\{\begin{pmatrix} \xi & 0 \\ 0 & \xi^{-1} \end{pmatrix}:\xi^3 = 1\right\} \bigcup \left\{\begin{pmatrix} 0 & \xi \\  \xi^{-1} & 0 \end{pmatrix}:\xi^3 = 1\right\}.\]
The irreducible representations of $C_3$ are all $1$-dimensional (and there are three of them), while $S_3$ has one $2$-dimensional irreducible representation, and two $1$-dimensional. By tensoring the irreducible representations of $C_3$ with those of $S_3$ we conclude that $H_4\simeq C_3\times S_3$ has three $2$-dimensional irreducible representations, and six $1$-dimensional ones. According to Theorem~\ref{thm:simples}, we obtain the following classification for the simple $\GL$-equivariant $\D$-modules on $X$:
\begin{itemize}
 \item $6$ modules with full support, corresponding to the $1$-dimensional representations of~$H_4$.
 \item $3$ modules with full support, corresponding to the $2$-dimensional irreducible representations of $H_4$.
 \item $1$ module with support $\ol{O_3}$.
 \item $3$ modules with support $\ol{O_2}$, corresponding to the $1$-dimensional representations of~$H_2$.
 \item $1$ module with support $O_0$.
\end{itemize}

Our next goal is to describe the characters of the $14$ simple modules listed above and establish some basic links between them. We have discussed the character of $E=\mc{L}(O_0;X)$ in (\ref{eq:character-E}). The characters of the $3$ modules with support $\ol{O_2}$ have been computed in \cite{raicu-veronese} and we start by recalling their description below. 

\subsection{$\D$-modules supported on $\ol{O_2}$ (the cone over the twisted cubic)}\label{subsec:Dmods-on-twistedcubic}

We consider the collection of integers $(\nu_i)_{i\in\bb{Z}}$ encoded by the generating function
\[ \sum_{i\in\bb{Z}} \nu_i \cdot t^i = \frac{1}{(1-t^2)(1-t^3)} = 1 + t^2 + t^3 + t^4 + t^5 + 2t^6 + t^7 + 2t^8 + 2t^9 + 2t^{10} + \cdots\]
and define for each dominant weight $\ll=(\ll_1,\ll_2)$ the integers
\begin{equation}\label{eq:mll-ell}
 m_{\ll} = \nu_{\ll_1-5} - \nu_{\ll_2-6}\mbox{ and }e_{\ll} = \langle [E],e^{\ll^{\vee}} \rangle \overset{(\ref{eq:wts-Fourier})}{=} \langle [S],e^{\ll-(6,6)} \rangle.
\end{equation}
For each $j=0,1,2$ we set 
\begin{equation}\label{eq:allj=0}
a_{\ll}^j = 0\mbox{ if }\ll_1 + \ll_2 \not\equiv j\ (\mbox{mod }3)
\end{equation}
and otherwise let
\begin{equation}\label{eq:allj-general}
a_{\ll}^j=
\begin{cases}
 m_{\ll}+e_{\ll} & \mbox{if }j=0; \\
 m_{\ll} & \mbox{if }j=1,2. \\
\end{cases}
\end{equation}
With this notation, \cite[Thm.~1.2]{raicu-veronese} asserts that the characters of the simple $\GL$-equivariant $\D$-modules $D_0,D_1,D_2$ supported on the twisted cubic are
\begin{equation}\label{eq:chars-Dj}
 [D_j] = \sum_{\ll} a_{\ll}^j \cdot e^{\ll^{\vee}} \mbox{ for }j=0,1,2.
\end{equation}
Notice the use of the dual weights $\ll^{\vee}$ in the above formula, which accounts for the fact that $e^{\ll}$ denotes the class of $\S_{\ll}V$ in $\Gamma(\GL)$ and not that of $\S_{\ll}W$! Notice also that (\ref{eq:allj=0}) and (\ref{eq:chars-Dj}) imply that
\begin{equation}\label{eq:wts-D012}
 \mbox{if }\langle [D_j],e^{\ll} \rangle \neq 0\mbox{ then }\ll_1 + \ll_2 + j \equiv 0\ (\mbox{mod }3).
\end{equation}

\begin{example}\label{ex:wts-D0}
 It will be useful later to know that 
 \begin{equation}\label{eq:wts-D0}
 \langle [D_0],e^{(-1,-5)} \rangle = \langle [D_0],e^{(-6,-9)} \rangle = 1 \mbox{ and } \langle [D_0], e^{(3,0)} \rangle = \langle [D_0], e^{(-2,-4)} \rangle = 0.
 \end{equation}
 This is equivalent to proving that $a^0_{(5,1)} = a^0_{(9,6)} = 1$ and $a^0_{(0,-3)} = a^0_{(4,2)} = 0$. We have that
 \[m_{(5,1)} = \nu_0 - \nu_{-5} = 1,\ m_{(9,6)} = \nu_4 - \nu_0 = 0,\ m_{(0,-3)} = \nu_{-5} - \nu_{-9} = 0\mbox{ and }m_{(4,2)} = \nu_{-1} - \nu_{-4} = 0.\]
 It follows from (\ref{eq:character-E}) that if $\langle [E],e^{(a,b)} \rangle \neq 0$ then $a,b\leq -6$, so
 \[ e_{(5,1)} = \langle [E],e^{(-1,-5)} \rangle = 0,\ e_{(0,-3)} = \langle [E],e^{(3,0)} \rangle = 0\mbox{ and }e_{(4,2)} = \langle [E],e^{(-2,-4)} \rangle = 0.\]
 We have moreover that 
 \[ e_{(9,6)} = \langle [S],e^{(3,0)} \rangle = \left\langle \frac{1+e^{(6,3)}}{(1-e^{(3,0)})(1-e^{(4,2)})(1-e^{(6,6)})}, e^{(3,0)} \right\rangle = 1.\]
 We conclude that 
 \[a^0_{(5,1)} = m_{(5,1)} + e_{(5,1)} = 1,\  a^0_{(9,6)} = m_{(9,6)} + e_{(9,6)} = 1,\mbox{ and }\]
 \[a^0_{(0,-3)} = m_{(0,-3)} + e_{(0,-3)} = 0,\mbox{ and }a^0_{(4,2)} = m_{(4,2)} + e_{(4,2)}=0,\]
 as desired.
\end{example}

\begin{example}\label{ex:wts-D12}
 It will also be useful later to know that 
 \begin{equation}\label{eq:wts-D12}
 \langle [D_2], e^{(-5,-9)} \rangle =  1 \mbox{ and } \langle [D_1], e^{(3,-1)}\rangle = 0.
 \end{equation}
 This is equivalent to proving that $a_{(9,5)}^1 = 1$ and $a_{(1,-3)}^2=0$, which in turn follow from
 \[ m_{(9,5)} = \nu_4 - \nu_{-1} = 1\mbox{ and }m_{(1,-3)} = \nu_{-5} - \nu_{-9} = 0.\]
\end{example}

\begin{lemma}\label{lem:maa}
 The values of $m_{(a,a)}$ when $a\in\bb{Z}$ are as follows: 
\begin{equation}\label{eq:maa}
 m_{(a,a)} = \begin{cases}
 -1 & \mbox{if }a \equiv 0\ (\mbox{mod }6),\ a\geq 6 \\
 1 & \mbox{if }a \equiv \pm 1\ (\mbox{mod }6),\ a\geq 5 \\
 0 & \rm{otherwise}.
 \end{cases}
\end{equation}
\end{lemma}

\begin{proof} Since $m_{(a,a)} = \nu_{a-5}-\nu_{a-6}$, it follows that $m_{(a,a)}=0$ for $a<5$. To compute $m_{(a,a)}$ for $a\geq 5$ we consider the generating function
\[\sum_{i\geq 0} (\nu_i - \nu_{i-1})\cdot t^i = (1-t)\cdot \sum_{i\geq 0} \nu_i\cdot t^i = \frac{1}{(1+t)(1-t^3)} = \frac{1-t+t^2}{1-t^6} = (1-t+t^2) + (t^6-t^7+t^8) + \cdots \]
from which the formula for $m_{(a,a)}$ follows.
\end{proof}

Combining Lemma~\ref{lem:maa} with (\ref{eq:chars-Dj}) we conclude that for $a\in\bb{Z}$ we have
\begin{equation}\label{eq:Dj-sl2-inv}
\langle D_1,e^{(a,a)}\rangle = \begin{cases}
 1 & \mbox{if }a \equiv 1\ (\mbox{mod }6),\ a\leq -5 \\
 0 & \rm{otherwise}.
 \end{cases}
 \qquad 
\langle D_2,e^{(a,a)}\rangle = \begin{cases}
 1 & \mbox{if }a \equiv -1\ (\mbox{mod }6),\ a\leq -7 \\
 0 & \rm{otherwise}.
 \end{cases}
\end{equation}
In particular, both $D_1, D_2$ contain invariant sections for the action of the special linear group $\SL \subset \GL$. In contrast, $D_0$ has no $\SL$-invariant sections as seen by the following.

\begin{lemma}\label{lem:D0-sl2-inv}
 For every integer $a\in\bb{Z}$ we have
\begin{equation}\label{eq:D0-sl2-inv}
 \langle [D_0], e^{(a,a)} \rangle = 0.
\end{equation}
\end{lemma}

\begin{proof}
 Using (\ref{eq:allj=0}), it suffices to consider the case when $a\in 3\bb{Z}$. Using (\ref{eq:allj-general}), and replacing $a$ by $-a$, it is then enough to show that $m_{(a,a)}+e_{(a,a)}=0$ for all $a\in 3\bb{Z}$ which follows from
 \[ m_{(a,a)} + e_{(a,a)} \overset{(\ref{eq:mll-ell})}{=} m_{(a,a)} + \langle [S],e^{(a-6,a-6)}\rangle \overset{(\ref{eq:character-S}),(\ref{eq:maa})}{=} \begin{cases}
 -1 + 1 = 0& \mbox{if }a \equiv 0\ (\mbox{mod }6),\ a\geq 6 \\
 0 + 0 = 0 & \mbox{otherwise}.
 \end{cases} 
 \qedhere\]
\end{proof}

Equipped with the description of the characters of $D_0,D_1,D_2$, we next prove that none of them appears as a composition factor in the localization $S_{\Delta}$. This will be important later in our analysis of the remaining simple $\D$-modules.

\begin{lemma}\label{lem:D12-notin-Sdelta}
 $D_1,D_2$ do not appear as composition factors of $S_{\Delta}$.
\end{lemma}

\begin{proof}
 It follows from (\ref{eq:character-Sdelta}) that if $\langle [S_{\Delta}],e^{\ll}\rangle \neq 0$ then $\ll_1 + \ll_2$ is divisible by $3$, so (\ref{eq:wts-D012}) implies that $\langle [D_1],e^{\ll}\rangle = \langle [D_2],e^{\ll}\rangle = 0$ showing that $D_1,D_2$ cannot occur as subrepresentations of $S_{\Delta}$ and therefore they are not composition factors of $S_{\Delta}$.
\end{proof}

\begin{lemma}\label{lem:D0-notin-Sdelta}
 $D_0$ does not appear as a composition factor of $S_{\Delta}$.
\end{lemma}

\begin{proof}
 As shown in Example~\ref{ex:wts-D0}, we have $\langle [D_0], e^{(-1,-5)} \rangle = 1$, so to prove the result it suffices to verify that 
 \begin{equation}\label{eq:wt-1-5-notSdelta}
 \langle [S_{\Delta}], e^{(-1,-5)}\rangle = 0.
 \end{equation}
 Assuming this isn't the case, it follows from (\ref{eq:character-Sdelta}) that we can find integers $0\leq a\leq 1$, $b,c\geq 0$ and $t\in\bb{Z}$ such that
 \begin{equation}\label{eq:wt-1-5}
 (-1,-5) = a\cdot(6,3) + b\cdot(3,0) + c\cdot(4,2) + t\cdot(6,6).
 \end{equation}
 Considering the difference between the first and second components of the above weights we get that
 \[ 4 = 3a + 3b + 2c\]
 which has a unique solution satisfying $0\leq a\leq 1$, $b,c\geq 0$, namely $a=b=0$ and $c=2$. The equation (\ref{eq:wt-1-5}) becomes $(-1,-5) = (8+6t,4+6t)$ which has no integer solution. We get a contradiction, showing (\ref{eq:wt-1-5-notSdelta}) and concluding our proof.
\end{proof}

\subsection{The remaining simple equivariant $\D$-modules}

We are left with describing the character of $P$, as well as those of the simple $\D$-modules with full support. We first identify $P$ as a composition factor of $S_{\Delta}$. 

\begin{lemma}\label{lem:P-in-Sdelta}
 The $\D$-module $P$ appears with multiplicity one as a composition factor of $S_{\Delta}$.
\end{lemma}

\begin{proof}
 Since $S_{\Delta}/S$ does not have full support, $S$ is the only composition factor of $S_{\Delta}$ with full support. It follows from the classification and Lemma~\ref{lem:D0-notin-Sdelta} that the remaining composition factors of $S_{\Delta}/S$ must be among $P,D_1,D_2,E$. Consider the dominant weight $(3,-3)$ and note that in order to compute $\langle [S_{\Delta}],e^{(3,-3)} \rangle$ we need to solve, as in the proof of Lemma~\ref{lem:D0-notin-Sdelta}, the equation
\begin{equation}\label{eq:wt3-3}
(3,-3) = a\cdot(6,3) + b\cdot(3,0) + c\cdot(4,2) + t\cdot(6,6),\mbox{ where }0\leq a\leq 1,\ b,c\geq 0\mbox{ and }t\in\bb{Z}.
\end{equation}
Every solution has $6=3a+3b+2c$, which in turn implies that $(a,b,c)$ is one of $(1,1,0)$, $(0,2,0)$, $(0,0,3)$. It follows that there is only one integer solution of (\ref{eq:wt3-3}), namely $(a,b,c,t) = (1,1,0,-1)$, and therefore
\[\langle [S_{\Delta}],e^{(3,-3)} \rangle = 1.\]
Using (\ref{eq:character-S}), (\ref{eq:character-E}) and (\ref{eq:wts-D012}) we conclude that
 \[ \langle [S],e^{(3,-3)} \rangle = \langle [E],e^{(3,-3)} \rangle =   \langle [D_1],e^{(3,-3)} \rangle =   \langle [D_2],e^{(3,-3)} \rangle =  0,\]
 so the only composition factor of $S_{\Delta}$ that may contain $\S_{(3,-3)}V$ is $P$. Since the multiplicity of $\S_{(3,-3)}V$ in $S_{\Delta}$ is $1$, the same must be true about the multiplicity of $P$ as a composition factor of $S_{\Delta}$.
\end{proof}

\begin{remark}\label{rem:P-in-Sdelta}
 Since $S_{\Delta}/S = H^1_{\ol{O_3}}(X,\mc{O}_X)$ it follows that $P=\mc{L}(\ol{O}_3,X)$ appears as a submodule of $S_{\Delta}/S$ and that the quotient $(S_{\Delta}/S)/P$ is supported on a proper closed subset of $\ol{O_3}$. Using Lemma~\ref{lem:D0-notin-Sdelta} and combining (\ref{eq:character-Sdelta}) with (\ref{eq:wts-D012}) we get that $D_0,D_1,D_2$ are not composition factors of $S_{\Delta}$, hence $(S_{\Delta}/S)/P$ must have support in $O_0$. It is not clear a priori that $(S_{\Delta}/S)/P\neq 0$, but as we'll see shortly we have that $(S_{\Delta}/S)/P=E$.
\end{remark}

\begin{lemma}\label{lem:FD0-full}
 The Fourier transform $\mc{F}(D_0)$ is different from $P,D_0,D_1,D_2,E$ and hence it is a simple $\GL$-equivariant $\D$-module with full support.
\end{lemma}

\begin{proof}
 Since the weight $\ll=(-1,-5)$ satisfies $\ll = \ll^{\vee} - (6,6)$, it follows that
 \[\langle [D_0],e^{(-1,-5)}\rangle \overset{(\ref{eq:wts-Fourier})}{=}  \langle [\mc{F}(D_0)],e^{(-1,-5)}\rangle \overset{(\ref{eq:wts-D0})}{=} 1.\] 
 Equation (\ref{eq:wt-1-5-notSdelta}) shows that $\mc{F}(D_0)$ does not appear as a composition factor of $S_{\Delta}$, hence $\mc{F}(D_0)\neq P$. We have
 \[\langle [\mc{F}(D_0)],e^{(3,0)}\rangle \overset{(\ref{eq:wts-Fourier})}{=} \langle [D_0],e^{(-6,-9)}\rangle \overset{(\ref{eq:wts-D0})}{=} 1,\]
and since $\langle [D_0],e^{(3,0)}\rangle = 0$ by (\ref{eq:wts-D0}), we conclude that $\mc{F}(D_0) \neq D_0$. The fact that $\mc{F}(D_0) \neq D_1,D_2$ follows from (\ref{eq:wts-D012}), and the fact that $\mc{F}(D_0)\neq E$ follows from $\langle [E],e^{(3,0)}\rangle = 0$ which is a consequence of (\ref{eq:character-E}).
\end{proof}

\begin{lemma}\label{lem:FD0-delta}
 The $\D$-module composition factors of $\mc{F}(D_0)_{\Delta}$ are among $D_0, P, \mc{F}(D_0)$. Moreover, both~$\mc{F}(D_0)$ and~$P$ appear as composition factors with multiplicity one.
\end{lemma}

\begin{proof}
We start by noting that since $\mc{F}(D_0)$ is simple and has full support, it is torsion free as an $S$-module. In particular $\Delta$ is a non-zerodivisor on $\mc{F}(D_0)$ and we get using Lemma~\ref{lem:localization} that
\[ [\mc{F}(D_0)_{\Delta}] = \lim_{n\to\infty} \left([\mc{F}(D_0)] \cdot e^{(-6n,-6n)}\right).\]
We also note that $\mc{F}(D_0)$ is a submodule of $\mc{F}(D_0)_{\Delta}$ and that the quotient $\mc{F}(D_0)_{\Delta}/\mc{F}(D_0)$ has support contained in $\ol{O_3}$, which means that the composition factors of $\mc{F}(D_0)_{\Delta}/\mc{F}(D_0)$ are among $P,D_0,D_1,D_2,E$, and in particular $\mc{F}(D_0)$ appears with multiplicity one.

Since $D_0$ has no $\SL$-invariant sections by Lemma~\ref{lem:D0-sl2-inv}, the same must be true about $\mc{F}(D_0)$ and thus also about $\mc{F}(D_0)_{\Delta}$. Since $E,D_1,D_2$ have $\SL$-invariant sections by (\ref{eq:character-E}) and (\ref{eq:Dj-sl2-inv}), it follows that they cannot appear as composition factors of $\mc{F}(D_0)_{\Delta}$, so the only remaining potential composition factors are~$P$ and~$D_0$. To conclude it remains to show that $P$ does appear as a composition factor and that it has multiplicity one.

We have
\[ \langle [\mc{F}(D_0)],e^{(-2,-4)} \rangle \overset{(\ref{eq:wts-Fourier})}{=} \langle [D_0],e^{(-2,-4)} \rangle \overset{(\ref{eq:wts-D0})}{=} 0, \]
so $D_0$ and $\mc{F}(D_0)$ do not contain $\S_{(-2,-4)}V$ as a subrepresentation. Moreover, using Lemma~\ref{lem:localization} we get
\[
\begin{aligned}
\langle [\mc{F}(D_0)_{\Delta}],e^{(-2,-4)} \rangle &= \lim_{n\to\infty} \langle [\mc{F}(D_0)],e^{(6n-2,6n-4)} \rangle = \lim_{n\to\infty} \langle [D_0],e^{(-2-6n,-4-6n)} \rangle \\
&=  \lim_{n\to\infty} (m_{(6n+4,6n+2)} + e_{(6n+4,6n+2)})
\end{aligned}
\]
We have $e_{(6n+4,6n+2)} = \langle [S], e^{6n-2,6n-4} \rangle = 1$ for all $n\geq 1$, and $m_{(6n+4,6n+2)} = \nu_{6n-1} - \nu_{6n-4}$ which can be computed from the generating function
\[ \sum_{i\geq 0} (\nu_i - \nu_{i-3})\cdot t^i = \frac{1-t^3}{(1-t^2)(1-t^3)} = \frac{1}{1-t^2} = 1+t^2 + t^4 + t^6 \cdots\]
Taking $i=6n-1$ we get that $\nu_i - \nu_{i-4} = 0$ and therefore $m_{(6n+4,6n+2)} = 0$ for all $n$. We conclude that 
\begin{equation}\label{eq:FD0-delta-2-4}
\langle [\mc{F}(D_0)_{\Delta}],e^{(-2,-4)} \rangle=1,
\end{equation}
so there is a composition factor of $\mc{F}(D_0)_{\Delta}$ containing $\S_{(-2,-4)}V$ as a subrepresentation. Since $D_0$ and $\mc{F}(D_0)$ do not have this property, this composition factor must then be $P$, occurring with multiplicity one as desired.
\end{proof}

\begin{corollary}\label{cor:factors-Sdelta}
 The module $P$ has no $\SL$-invariant sections, and is therefore not isomorphic to $S_{\Delta}/S$. In fact, $S_{\Delta}/S$ has length two as a $\D$-module, with composition factors $P$ and $E$.
\end{corollary}

\begin{proof}
 It follows from Lemma~\ref{lem:FD0-delta} that $P$ is a composition factor of $\mc{F}(D_0)_{\Delta}$. Since $\mc{F}(D_0)_{\Delta}$ has no $\SL$-invariant sections, the same must be true about $P$, but
 \[\langle [S_{\Delta}/S], e^{(-6,-6)}\rangle = \langle [S_{\Delta}], e^{(-6,-6)}\rangle - \langle [S], e^{(-6,-6)}\rangle \overset{(\ref{eq:character-S}),(\ref{eq:character-Sdelta})}{=} 1 - 0 = 1,\]
 and therefore $P\neq S_{\Delta}/S$. Using Remark~\ref{rem:P-in-Sdelta} and the fact that $\langle [S_{\Delta}/S], e^{(-6,-6)}\rangle = \langle [E], e^{(-6,-6)}\rangle = 1$ we conclude that $(S_{\Delta}/S)/P = E$.
\end{proof}

As a consequence of Corollary~\ref{cor:factors-Sdelta} and using (\ref{eq:character-S}), (\ref{eq:character-Sdelta}), (\ref{eq:character-E}), we determine the character of $P$ via
\[[P] = [S_{\Delta}] - [S] - [E].\]

\begin{corollary}\label{cor:factors-FD0-delta}
 The $\D$-module $\mc{F}(D_0)_{\Delta}$ has length three, with composition factors $D_0, P, \mc{F}(D_0)$, each appearing with multiplicity one.
\end{corollary}

\begin{proof} We already know that $P$ and $\mc{F}(D_0)$ appear as composition factors with multiplicity one, so we need to show that the same is true for $D_0$. We consider the weight $\ll=(-6,-9)$ and observe using (\ref{eq:character-S}), (\ref{eq:character-Sdelta}) and (\ref{eq:character-E}) that
\[ \langle [S], e^{(-6,-9)} \rangle = 0\quad\mbox{and}\quad\langle [S_{\Delta}], e^{(-6,-9)} \rangle = \langle [E], e^{(-6,-9)} \rangle = 1\]
which based on Corollary~\ref{cor:factors-Sdelta} implies that
\[ \langle [P], e^{(-6,-9)} \rangle = \langle [S_{\Delta}], e^{(-6,-9)} \rangle - \langle [S], e^{(-6,-9)} \rangle - \langle [E], e^{(-6,-9)} \rangle = 0.\]
We have
\[ \langle [D_0],e^{(-6,-9)} \rangle \overset{(\ref{eq:wts-D0})}{=} 1 \mbox{ and } \langle [\mc{F}(D_0)],e^{(-6,-9)} \rangle \overset{(\ref{eq:wts-Fourier})}{=} \langle [D_0],e^{(3,0)} \rangle \overset{(\ref{eq:wts-D0})}{=} 0.\]
Since $P$ and $\mc{F}(D_0)$ do not contain $\S_{(-6,-9)}V$ as a subrepresentation, and $D_0$ contains it with multiplicity one, it follows that the multiplicity of $D_0$ as a composition factor of $\mc{F}(D_0)_{\Delta}$ is equal to 
\[\langle [\mc{F}(D_0)_{\Delta}],e^{(-6,-9)} \rangle = \lim_{n\to\infty} \langle [\mc{F}(D_0)],e^{(6n-6,6n-9)} \rangle = \lim_{n\to\infty} \langle [D_0],e^{(3-6n,-6n)} \rangle =  \lim_{n\to\infty} (m_{(6n,6n-3)} + e_{(6n,6n-3)})\]
We have $e_{(6n,6n-3)} = \langle [S], e^{6n-6,6n-9} \rangle = 1$ for all $n\geq 1$, and $m_{(6n,6n-3)} = \nu_{6n-5} - \nu_{6n-9}$ which can be computed from the generating function
\[ \sum_{i\geq 0} (\nu_i - \nu_{i-4})\cdot t^i = \frac{1-t^4}{(1-t^2)(1-t^3)} = \frac{1+t^2}{1-t^3} = (1+t^2) + (t^3 + t^5) + (t^6 + t^8) + \cdots\]
Taking $i=6n-5$ we get that $\nu_i - \nu_{i-4} = 0$ and therefore $m_{(6n,6n-3)} = 0$ for all $n$. We conclude that $\langle [\mc{F}(D_0)_{\Delta}],e^{(-6,-9)} \rangle = 1$, so that $D_0$ appears with multiplicity one as a composition factor of $\mc{F}(D_0)_{\Delta}$, concluding the proof. 
\end{proof}

\begin{lemma}\label{lem:FD12-full}
 The Fourier transforms $\mc{F}(D_1)$ and $\mc{F}(D_2)$ are simple $\GL$-equivariant $\D$-modules with full support.
\end{lemma}

\begin{proof}
 Just as in the proof of Lemma~\ref{lem:FD0-full}, it suffices to verify that $\mc{F}(D_1)$ and $\mc{F}(D_2)$ are not among the modules $P,D_0,D_1,D_2,E$. It follows from (\ref{eq:wts-Fourier}) and (\ref{eq:wts-D012}) that
 \[ \mbox{if }\langle [\mc{F}(D_j)],e^{\ll} \rangle \neq 0\mbox{ then }\ll_1 + \ll_2 \equiv j \ (\mbox{mod }3),\]
 so we only need to verify that $\mc{F}(D_1)\neq D_2$ and $\mc{F}(D_2)\neq D_1$. Since $\mc{F}$ is an involution, these two assertions are equivalent, so we only verify the first one. We have
 \[ \langle [\mc{F}(D_2)], e^{(3,-1)}\rangle \overset{(\ref{eq:wts-Fourier})}{=} \langle [D_2], e^{(-5,-9)} \rangle \overset{(\ref{eq:wts-D12})}{=} 1\mbox{ and }\langle [D_1], e^{(3,-1)} \rangle \overset{(\ref{eq:wts-D12})}{=} 0\]
 so $\mc{F}(D_2)\neq D_1$, concluding the proof.
\end{proof}

We write $Q_0 = \mc{F}(D_0)$ and let $Q_j = (Q_0)_{\Delta}\cdot \Delta^{j/3}$ for $j=1,2$. We have that $Q_0,Q_1,Q_2$ are $\D$-modules with full support and no $\SL$-invariant sections, and $Q_0$ is simple because the Fourier transform preserves simplicity. Our aim is to show that $Q_1$ and $Q_2$ are also simple, but for now we only show the following.

\begin{lemma}\label{lem:Q12-factors}
 All of the composition factors of $Q_1,Q_2$ have full support, or equivalently, none of the modules $E,P,D_0,D_1,D_2$ appear as composition factors in $Q_1,Q_2$. 
\end{lemma}

\begin{proof} Since $[\bb{C}\cdot\Delta^{j/3}] = e^{(2j,2j)}$ it follows that
 \begin{equation}\label{eq:wts-Qj}
  \mbox{if }\langle [Q_j],e^{\ll} \rangle \neq 0\mbox{ then }\ll_1 + \ll_2 \equiv j \ (\mbox{mod }3),
 \end{equation}
so it suffices to check that $D_1$ is not a composition factor of $Q_2$, and that $D_2$ is not a composition factor of~$Q_1$. To see this it suffices to observe that by (\ref{eq:Dj-sl2-inv}) the modules $D_1,D_2$ have non-zero $\SL$-invariant sections, whereas $Q_1$ and $Q_2$ do not. 
\end{proof}

We define as in the introduction $F_i = S_{\Delta}\cdot \Delta^{i/6}$ for $i=-1,0,1,2,3,4$, and let $G_i \subseteq F_i$ denote the $\D$-submodule generated by $\Delta^{i/6}$.

\begin{lemma}\label{lem:G0234}
 We have that $G_i = F_i$ for $i=2,3,4$. Moreover, $G_i$ is a simple $\D$-module for $i=0,2,3,4$.
\end{lemma}

\begin{proof}We have that $G_0 = S = \mc{L}(X;X)$ is a simple $\D$-module, so we only need to focus on the cases $i=2,3,4$. Recall from \cite[Example~9.1]{b-functions} that the $b$-function of $\Delta$ is 
\[b_{\Delta}(s) = (s+1)^2\cdot(s+5/6)\cdot(s+7/6)\]
hence no integer translate of $i/6$ is a root of $b_{\Delta}(s)$ for $i=2,3,4$. The argument in \cite[Section~7]{raicu-matrices} with $f=\Delta$ applies to show that $G_i = F_i$ for $i=2,3,4$, and that each such $G_i$ is a simple $\D$-module.
\end{proof}

\begin{corollary}\label{cor:simples-full-support}
 The 9 simple equivariant $\D$-modules on $X = \Sym^3 W$ having full support are
 \[G_0,G_2,G_3,G_4,\mc{F}(D_1),\mc{F}(D_2), Q_0,Q_1,Q_2.\]
Moreover, we have that $G_1=\mc{F}(D_2)$, $G_{-1}=\mc{F}(D_1)$, $D_1 = F_1/G_1$ and $D_2 = F_{-1}/G_{-1}$.
\end{corollary}

\begin{proof}
 We have already seen that $G_0,G_2,G_3,G_4,\mc{F}(D_1),\mc{F}(D_2),Q_0=\mc{F}(D_0)$ are simple. For $j=1,2$ let $Q_j'$ denote any of the composition factors of $Q_j$, which by Lemma~\ref{lem:Q12-factors} has full support. It follows from (\ref{eq:wts-Qj}) that the modules $Q_0,Q_1',Q_2'$ are mutually distinct. Moreover, they are distinct from $G_0,G_2,G_3,G_4,\mc{F}(D_1),\mc{F}(D_2)$ since the latter contain non-zero $\SL$-invariant sections, whereas $Q_0,Q_1',Q_2'$ do not. It follows that
 \begin{equation}\label{eq:simples-full-support}
 G_0,G_2,G_3,G_4,\mc{F}(D_1),\mc{F}(D_2), Q_0,Q_1',Q_2'
 \end{equation}
 are all the 9 simple equivariant $\D$-modules on $X = \Sym^3 W$ having full support. Since for $j=1,2$ the modules $Q_j'$ were chosen arbitrarily among the composition factors of $Q_j$, and since the list (\ref{eq:simples-full-support}) is independent on these choices, we conclude that in fact each $Q_j$ only has one composition factor (possibly with multiplicity bigger than one). We have seen in (\ref{eq:FD0-delta-2-4}) that $\langle [Q_0],e^{(-2,-4)}\rangle = 1$, which implies that $\langle [Q_1],e^{(0,-2)}\rangle = 1$ and $\langle [Q_2],e^{(2,0)}\rangle = 1$, and thus the composition factors of $Q_1$ and $Q_2$ cannot appear with multiplicity bigger than one. This means that $Q_1' = Q_1$ and $Q_2'=Q_2$, proving our first assertion.
 
The $\D$-modules $F_1$ and $F_{-1}$ are torsion free, so they contain a simple submodule with full support. Looking at the congruence of $\ll_1 + \ll_2$ modulo $3$ for the representations $\S_{\ll}V$ contained in $F_1$ and $F_{-1}$, and comparing with the list of simple equivariant $\D$-modules that we obtained, we conclude that
\begin{itemize}
 \item $\mc{F}(D_2)$ and $Q_2$ are the only possible submodules of $F_1$ with full support, and $D_1$ is the only possible composition factor of $F_1$ that does not have full support.
 \item $\mc{F}(D_1)$ and $Q_1$ are the only possible submodules of $F_{-1}$ with full support, and $D_2$ is the only possible composition factor of $F_{-1}$ that does not have full support.
\end{itemize}

Since $F_1 = S_{\Delta}\cdot\Delta^{1/6}$ it follows from (\ref{eq:character-Sdelta}) that
\begin{equation}\label{eq:character-F1}
 [F_1] = \frac{1+e^{(6,3)}}{(1-e^{(3,0)})(1-e^{(4,2)})} \cdot e^{(6,6)\bb{Z}+(1,1)}
\end{equation}
and in particular $\langle [F_1],e^{(a+2,a)} \rangle \neq 0$ if and only if $a \equiv 3\ (\mbox{mod }6)$. This shows that $\langle [F_1],e^{(2,0)} \rangle = 0$, and since $\langle [Q_2],e^{(2,0)}\rangle = 1$ we conclude that $Q_2$ is not a submodule of $F_1$. It follows that $\mc{F}(D_2)$ is a submodule of $F_1$, and in fact $\mc{F}(D_2)$ has multiplicity one as a composition factor of $F_1$, due to the fact that
\[\langle [\mc{F}(D_2)], e^{(1,1)} \rangle \overset{(\ref{eq:wts-Fourier}),(\ref{eq:Dj-sl2-inv})}{=} 1\quad\mbox{ and }\quad \langle [F_1], e^{(1,1)} \rangle = 1.\]
Since $\bb{C}\cdot\Delta^{1/6}$ is the unique copy of $\S_{(1,1)}V$ inside $F_1$, it follows that $\Delta^{1/6}\in \mc{F}(D_2)$. Since $\mc{F}(D_2)$ is simple, it is generated by $\Delta^{1/6}$ as a submodule of $F_1$, hence $\mc{F}(D_2) = G_1$. Now the quotient $F_1/G_1$ can only have $D_1$ as a composition factor, and $D_1$ will appear with multiplicity one since
\[ \langle [F_1], e^{(-5,-5)} \rangle \overset{(\ref{eq:character-F1})}{=} 1 \overset{(\ref{eq:Dj-sl2-inv})}{=}  \langle [D_1], e^{(-5,-5)} \rangle\quad\mbox{and}\quad\langle [G_1], e^{(-5,-5)} \rangle = \langle [\mc{F}(D_2)], e^{(-5,-5)} \rangle \overset{(\ref{eq:wts-Fourier}),(\ref{eq:Dj-sl2-inv})}{=} 0.\]
This shows that $F_1/G_1 = D_1$, as desired. The proof that $G_{-1} = \mc{F}(D_1)$ and that $F_{-1}/G_{-1}$ is completely analogous, and we leave the details to the interested reader.
\end{proof}

\begin{remark}
 Since the modules $G_i$, $i=-1,0,1,2,3,4$, have holonomic rank one (when restricted to the dense orbit $O_4$ they are isomorphic to $F_i$, which are free $S_{\Delta}$-modules of rank one), they correspond via the classification in Section~\ref{subsec:classification} to the $1$-dimensional representations of the component group $H_4\simeq C_3\times S_3$. The remaining simple $\D$-modules with full support, $Q_0,Q_1,Q_2$, correspond to the $2$-dimensional representations of $H_4$, and in particular they have holonomic rank two, which is perhaps not immediately clear from their construction.
\end{remark}

At this point we can describe the characters of all the remaining simple $\D$-modules. Using (\ref{eq:chars-Dj}) and (\ref{eq:wts-Fourier}) we obtain from Corollary~\ref{cor:simples-full-support} explicit formulas for $[G_1] = \mc{F}[D_2]$ and $[G_{-1}] = \mc{F}[D_1]$. Since $[G_i] = [S_{\Delta}] \cdot e^{(i,i)}$ for $i=2,3,4$, they can be computed from (\ref{eq:character-Sdelta}), while the character of $G_0=S$ has been determined in (\ref{eq:character-S}). The character of $Q_0$ is $[Q_0]=\mc{F}[D_0]$, while the characters of $Q_1,Q_2$ are determined by $[Q_i] = [Q_0]\cdot e^{(2i,2i)}$ for $i=1,2$.

We end this section by analyzing some non-simple $\D$-modules. In Section~\ref{sec:category-modDx} we will study from a quiver perspective the indecomposable objects in the category of equivariant $\D$-modules. Two such example is given below.

\begin{lemma}\label{lem:indQ0-del}
The module $D_0$ is not isomorphic to a submodule of $(Q_0)_\Delta / Q_0$, and in particular $(Q_0)_\Delta / Q_0$ is indecomposable.
\end{lemma}

\begin{proof}
Recall from Corollary \ref{cor:factors-FD0-delta} that $(Q_0)_\Delta/Q_0$ has length two with composition factors $P$ and $D_0$. If it were decomposable, it would be isomorphic to $P\oplus D_0$, and in particular $D_0$ would be a submodule of  $(Q_0)_\Delta / Q_0$. It is then enough to prove that this is not the case.

Assume by contradiction that $D_0\subset (Q_0)_\Delta/Q_0$, let $\lambda= (-6,-9)$ and recall that $\langle [D_0],e^{\lambda} \rangle \overset{(\ref{eq:wts-D0})}{=} 1$. Using (\ref{eq:allj-general}) we obtain
\begin{equation}\label{eq:wts-3-6D0}
 \langle [D_0], e^{(-3,-6)}\rangle= 0 \mbox{ and } \langle [D_0], e^{(0,-3)} \rangle = 0.
\end{equation}
 
Let $v_\lambda\in D_0 \subset (Q_0)_\Delta/Q_0$ be a highest weight vector of weight $\lambda$ and choose a lift of $v_\lambda$ to a highest weight vector $\tilde{v}_\lambda \in (Q_0)_\Delta$. Since $\Delta$ has weight $(6,6)$ and $\langle [D_0], e^{\ll+(6,6)} \rangle = \langle [D_0], e^{(0,-3)} \rangle = 0$, we conclude that $\Delta \cdot v_\lambda = 0$ and therefore $\Delta \cdot \tilde{v}_\lambda \in Q_0$. We have however that
\[\langle [Q_0], e^{\ll+(6,6)} \rangle = \langle [Q_0], e^{(0,-3)} \rangle  \overset{(\ref{eq:wts-Fourier})}{=}  \langle [D_0], e^{(-3,-6)} \rangle \overset{(\ref{eq:wts-3-6D0}) }{=} 0\]
which implies that $\Delta \cdot \tilde{v}_{\ll}=0$. This is a contradiction since $\Delta$ is a non-zero divisor on $(Q_0)_\Delta$, concluding our proof.
\end{proof}

\begin{lemma}\label{lem:indS-del}
The module $E$ is not isomorphic to a submodule of $S_\Delta / S$, and in particular $S_\Delta / S$ is indecomposable.
\end{lemma}

\begin{proof}
 By Corollary~\ref{cor:factors-Sdelta}, $S_{\Delta}/S$ has length two, with composition factors $P$ and $E$, so we can prove our result using the same strategy as in Lemma~\ref{lem:indQ0-del}. For that it suffices to take again $\ll = (-6,-9)$ and use
 \[ \langle [E], e^{\ll}\rangle = 1,\mbox{ and } \langle [E], e^{\ll+(6,6)}\rangle = \langle [S], e^{\ll+(6,6)}\rangle = 0.\qedhere\]
\end{proof}

\section{The category of equivariant coherent $\D$-modules}\label{sec:category-modDx}

In this section we continue to denote by $W$ a 2-dimensional complex vector space, by $\GL=\GL(W)$ the group of invertible linear transformations of $W$, and by $X=\Sym^3 W$ the space of binary cubic forms which is equipped with a natural action of $\GL$. The goal of this section is to describe the category of $\GL$-equivariant coherent $\D_X$-modules as the category of finite-dimensional representations of a quiver (as mentioned in Section \ref{subsec:Dmods}). Then we identify (up to isomorphism) all the indecomposable representations of the quiver, so implicitly, all the indecomposable equivariant $\D_X$-modules.

\subsection{The quiver description of the category of equivariant $\D$-modules}

We now proceed to determine the quiver $(\Q,\I)$ for the category of $\GL$-equivariant coherent $\D_X$-modules. We will use freely the following general properties of the correspondence between the category $\opmod_{\GL}(\D_X)$ and the quiver $(\Q,\I)$ representing~it:
\begin{itemize}
 \item There is a bijective correspondence between simples $M\in\opmod_{\GL}(\D_X)$ and nodes $m$ of $\Q$ -- using notation (\ref{eq:defS}), the representation of $(\Q,\I)$ corresponding to $M$ is $\sS^m$.
 \item If $D\in\opmod_{\GL}(\D_X)$ corresponds to $\sV^D\in\rep(\Q,\I)$ then for every simple $M\in\opmod_{\GL}(\D_X)$ corresponding to a node $m$ we have that the multiplicity of $M$ as a composition factor of $D$ is equal to $\dim\sV^D_m$.
 \item Specializing the remark above to the case when $D$ is the injective envelope of the simple $M$, we have that $\sV^D = \sI^m$. Using (\ref{eq:defP-I}), it follows that for every simple $N$ with corresponding node $n$, the multiplicity of $N$ as a composition factor of $D$ is equal to the number of paths in $\Q$ (modulo relations) from $n$ to $m$.
 \item Dually, if $D$ is the projective cover of a simple $M$ then the multiplicity of a simple $N$ as a composition factor of $D$ is equal to the number of paths from $m$ to $n$.
 \item If $m,n$ are nodes of $\Q$ corresponding to the simples $M,N$, then the number of arrows from $m$ to $n$ equals $\dim_{\bb{C}} \Ext^1(M,N)$ \cite[Chapter III, Lemma 2.12]{ass-sim-sko}.
 \item The holonomic duality functor $\bb{D}$ (resp. the Fourier transform $\mc{F}$) described in Section~\ref{subsec:Dmods} induces a self-duality (resp. a self-equivalence) of the category $\rep(\Q,\I)$. If $m,n$ are nodes of $\Q$ corresponding to the simples $M,N$, then we write $n=\bb{D}(m)$ (resp. $n=\mc{F}(m)$) when $N=\bb{D}(M)$ (resp. $N = \mc{F}(M)$). The number of arrows from $m$ to $n$ is equal to the number of arrows from $\bb{D}(n)$ to $\bb{D}(m)$, as well as to the number of arrows from $\mc{F}(m)$ to $\mc{F}(n)$. An analogous relationship holds for the number of paths from $m$ to $n$ (considered as always modulo the relations in $\I$).
\end{itemize}

Based on these general remarks we can prove the following.

\begin{lemma}\label{lem:MN'N''}
 Suppose that we have an exact sequence in $\opmod_{\GL}(\D_X)$
 \begin{equation}\label{eq:sesM-IM-N}
 0 \lra M \lra I^M \lra N \lra 0
 \end{equation}
where $I^M$ denotes the injective envelope of the simple module $M$, and $N$ has length two, with composition factors $N',N''$. We assume that $M,N',N''$ are pairwise distinct and let $m,n',n''$ denote the corresponding nodes of $\Q$. If $N''$ is not isomorphic to a submodule of $N$ then there exists a unique arrow in $\Q$ from $n'$ to $m$, and no arrow from $n''$ to $m$.
\end{lemma}

\begin{proof} Since the socle of $I^M$ is $M$, we get $\Hom(N',I^M) = \Hom(N'',I^M) = 0$. Moreover, the assumptions on $N$ show that $N'$ must be a submodule of $N$, with $N'' = N/N'$, which in turn implies that $\Hom(N',N) = \bb{C}$. The long exact sequences associated to $\Hom(N',\bullet)$ and $\Hom(N'',\bullet)$ applied to (\ref{eq:sesM-IM-N}) yield
\[0 \lra \Hom(N',N) \lra \Ext^1(N',M) \lra \Ext^1(N',I^M) = 0,\]
\[0 \lra \Hom(N'',N) \lra \Ext^1(N'',M) \lra \Ext^1(N'',I^M) = 0,\]
where the vanishing of $\Ext^1(\bullet,I^M)$ follows from the fact that $I^M$ is injective. We get that $\Ext^1(N',M)=\bb{C}$ and $\Ext^1(N'',M)=0$, i.e. there exists exactly one arrow from $n'$ to $m$, and no arrow from $n''$ to $m$.
\end{proof}

%
%

We are now ready to prove the main theorem regarding the quiver description of $\opmod_{\GL}(\D_X)$.

\begin{theorem}\label{thm:quiver}
There is an equivalence of categories 
\[\opmod_{\GL}(\D_X) \simeq \rep(\Q,\I),\]
where $\rep(\Q,\I)$ is the category of finite-dimensional representations of a quiver $\Q$ with relations $\I$, described as follows. The vertices and arrows of the quiver $\Q$ are depicted in the diagram
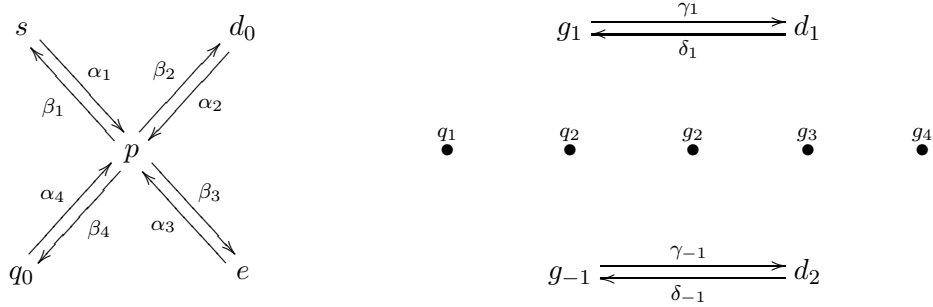
\[\xymatrix@=2.3pc@L=0.2pc{
s \ar@<0.5ex>[dr]^{\alpha_1} & & d_0 \ar@<0.5ex>[dl]^{\alpha_2}  & & & g_{1} \ar@<0.5ex>[rr]^{\gamma_1} & & d_1 \ar@<0.5ex>[ll]^{\delta_1} & \\
 & p \ar@<0.5ex>[ul]^{\beta_1} \ar@<0.5ex>[ur]^{\beta_2}\ar@<0.5ex>[dl]^{\beta_4} \ar@<0.5ex>[dr]^{\beta_3} & & & \overset{q_1}{\bullet} & \overset{q_2}{\bullet} & \overset{g_2}{\bullet} & \overset{g_3}{\bullet} & \overset{g_4}{\bullet} \\
q_0 \ar@<0.5ex>[ur]^{\alpha_4} & & e \ar@<0.5ex>[ul]^{\alpha_3} & & & g_{-1} \ar@<0.5ex>[rr]^{\gamma_{-1}} & & d_2 \ar@<0.5ex>[ll]^{\delta_{-1}} &
}\]
and the set of relations $\I$ is given by all 2-cycles and all non-diagonal compositions of two arrows:
\[\alpha_i \beta_i \mbox{ and } \beta_i\alpha_i \mbox{ for } i=1,2,3,4,\quad \gamma_i\delta_i\mbox{ and }\delta_i\gamma_i\mbox{ for }i=1,-1,\mbox{ and}\]
\[\alpha_1 \beta_2, \alpha_1\beta_4 , \alpha_2\beta_1, \alpha_2\beta_3, \alpha_3\beta_2,\alpha_3\beta_4,\alpha_4\beta_1,\alpha_4\beta_3.\]
\end{theorem}

\begin{proof}
Let $M$ be any of the simples $G_2,G_3,G_4,Q_1,Q_2$, and consider the open embedding $j:O_4 \to X$. We have by construction that $j_*j^*M =M_{\Delta}= M$, hence $M$ is injective by Lemma \ref{lem:inj}. The holonomic duality functor $\bb{D}$ takes a simple $\D$-module corresponding to a local system on $U$ to the simple corresponding to the dual local system and therefore
\[G_2 \overset{\bb{D}}{\llra}G_4,\ \bb{D}(G_3)=G_3,\ Q_1 \overset{\bb{D}}{\llra}Q_2.\]
Since $\bb{D}$ takes injectives to projectives, this shows that $M$ is both injective and projective, hence $\Ext^1(M,N) = \Ext^1(N,M)=0$ for all $N\in\opmod_{\GL}(\D_X)$. We conclude that $g_2,g_3,g_4,q_1,q_2$ are isolated nodes of the quiver~$\Q$.

Consider next the modules $G_i$ with $i=\pm 1$. We have by Lemma~\ref{lem:inj} that $j_*j^* G_{i} = F_{i}$ is the injective envelope of $G_i$. Combining this with Corollary \ref{cor:simples-full-support} we get non-split exact sequences
\[0 \to G_{1} \to F_{1} \to D_1\to 0 \mbox{ and } 0 \to G_{-1} \to F_{-1} \to D_2\to 0\] 
so $\Ext^1(D_2,G_{-1}) = \Ext^1(D_1,G_{1}) = \bb{C}$. It follows that in $\Q$ there exist unique arrows $\delta_1$ from $d_1$ to $g_1$, and $\delta_{-1}$ from $d_2$ to $g_{-1}$. Moreover, since $F_i$ is the injective envelope of $G_i$ there are no other paths of positive length (and in particular no arrows) with target $g_1$ or $g_{-1}$. Using
\[D_1 \overset{\bb{D}}{\llra}D_2,\ G_1 \overset{\bb{D}}{\llra}G_{-1},\]
we get unique arrows $\gamma_1$ from $g_1=\bb{D}(g_{-1})$ to $d_1=\bb{D}(d_2)$, and $\gamma_{-1}$ from $g_{-1}=\bb{D}(g_1)$ to $d_2=\bb{D}(d_1)$, and moreover there are no other paths with source in $g_1$ or $g_{-1}$. Applying next the Fourier transform
\[G_1 \overset{\mc{F}}{\llra}D_2,\ G_{-1} \overset{\mc{F}}{\llra}D_1,\]
we conclude that there are unique arrows in and out of $d_1,d_2$ (namely $\gamma_{\pm 1},\delta_{\pm 1}$) and no other paths of positive length in or out of $d_1,d_2$. This forces the relations $\delta_i\gamma_i=\gamma_i\delta_i=0$ for $i=\pm 1$.

We are left to consider the restriction of the quiver to the nodes $e,d_0,p,q_0,s$. Since the modules $E,D_0,P,Q_0,S$ are preserved by holonomic duality, it follows that for $m,n\in\{e,d_0,p,q_0,s\}$, the number of arrows from $m$ to $n$ is equal to the number of arrows from $n$ to $m$.

Applying Lemma~\ref{lem:inj} with $Y=X$, $O=O_3$ and $M=S$ we find that $S_{\Delta}$ is the injective envelope of $S$, and combining Lemmas~\ref{lem:MN'N''} and~\ref{lem:indS-del} we conclude that there exists a unique arrow from $p$ to $s$, which we denote by $\b_1$, and no arrow from $e$ to $s$. Applying the duality $\bb{D}$ we find a unique arrow $\a_1$ from $s$ to $p$. Since $D_0$ and $Q_0$ do not appear as composition factors of $S_{\Delta}$, there is no path (and thus no arrow) from either $d_0$ or $q_0$ to $s$. We can next argue in the similar fashion by taking $M=Q_0$ and appealing to Lemma~\ref{lem:indQ0-del} to conclude that there are unique arrows $\b_4$ from $p$ to $q_0$ and $\a_4$ from $q_0$ to $p$, there is no arrow from $d_0$ to $q_0$, and no path from either $e$ or $s$ to $q_0$.

Since the Fourier transform fixes $P$ and swaps the elements in each pair $(E,S)$ and $(Q_0,D_0)$, it follows from the conclusions of the preceding paragraph that there exists a pair of arrows $\a_3,\b_3$ between $p$ and $e$, and another pair of arrows $\a_2,\b_2$ between $p$ and $d_0$. Moreover, there are no paths from either $d_0$ or $q_0$ to $e$, and no paths from either $e$ or $s$ to $d_0$, so in particular we have obtained a complete list of arrows in $\Q$. 

The relations $\a_1\b_1=\a_4\b_4=0$ follow from the fact that $S$ and $Q_0$ appear with multiplicity one inside their respective injective envelopes. Applying the duality $\bb{D}$ we conclude that $\b_1\a_1=\b_4\a_4=0$, and further applying the Fourier transform we get $\a_2\b_2=\b_2\a_2=\a_3\b_3=\b_3\a_3=0$. Since there are no paths between either of $e,s$ and either $q_0,d_0$ we conclude that 
\[\alpha_1 \beta_2 = \alpha_1\beta_4 = \alpha_2\beta_1 =\alpha_2\beta_3 = \alpha_3\beta_2 = \alpha_3\beta_4 = \alpha_4\beta_1 = \alpha_4\beta_3 = 0.\]
Since $E$ appears in the injective envelope of $S$, there has to be a path from $e$ to $s$ which is necessarily $\a_3\b_1$. By holonomic duality, there is a path from $s$ to $e$, namely $\a_1\b_3$. Similarly, since $D_0$ appears in the injective envelope of $Q_0$ we get the path $\a_2\b_4$ from $d_0$ to $q_0$, and by duality we have the path $\a_4\b_2$ from $q_0$ to $d_0$. This shows that there are no more relations in the quiver $\Q$, concluding our proof.
\end{proof}

\begin{remark}
As any $\D$-module in $\opmod_{\GL}(\D_X)$ has a projective (resp. injective) resolution, the indecomposable projective (resp. injective) $\D$-modules play a special role. We can construct all the indecomposable projective (resp. injective) $\D$-modules in $\opmod_{\GL}(\D_X)$ using $\D$-module-theoretic tools as follows. Up to duality $\bb{D}$, it is enough to describe the injectives. We already obtained in the proof of Theorem \ref{thm:quiver} the injective envelopes for the simples with full support. Namely, if $M$ is a simple equivariant $\D$-module with full support, then its injective envelope is $M_\Delta$ (see Lemma \ref{lem:inj}). By taking Fourier transform (which preserves injectives), the injective envelopes of $E,D_0,D_1,D_2$ are $\mc{F}(S_\Delta), \mc{F}((Q_0)_\Delta),\mc{F}(F_{-1}),\mc{F}(F_1)$, respectively. We are left to identify the injective envelope of $P$. One possible description is as follows.  Let $U=X\backslash \ol{O_2}$ and $j: U \to X$ the open embedding. Since $O_3$ is smooth and closed in $U$ of codimension 1, the sheaf $\mc{H}^1_{O_3} (\mc{O}_{U})$ is the simple $\D_{U}$-module $j^*P$, and $\mc{H}^i_{O_3}(\mc{O}_U)=0$ for $i\neq 1$. We have $j_{*} \mc{H}^1_{O_3} (\mc{O}_{U})= H^0 \mc{H}^1_{O_3} (\mc{O}_{U}) = H^1_{O_3}(X,\mc{O}_X)$, since the spectral sequence $H^p(\mc{H}^q_{O_3}(U,\mc{O}_U))\Rightarrow H^{p+q}_{O_3}(X,\mc{O}_X)$ degenerates. By Lemma \ref{lem:inj}, this shows that  $H=H^1_{O_3}(X,\mc{O}_X)$ is the injective envelope of $P$ in the subcategory $\opmod_G^{\ol{O_3}}(\D_X)$. By Theorem \ref{thm:quiver}, the $\D$-module $H$ and its Fourier transform $\mc{F}(H)$ will correspond to the following representations of the quiver $(\Q,\I)$:
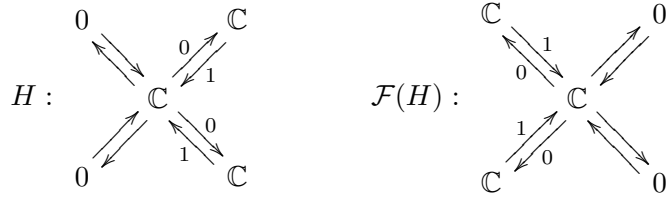
\[H: \begin{aligned}\xymatrix@=1.3pc@L=0.1pc{
0 \ar@<0.5ex>[dr] & & \bb{C} \ar@<0.5ex>[dl]^{1} \\
 & \bb{C} \ar@<0.5ex>[ul] \ar@<0.5ex>[ur]^{0}\ar@<0.5ex>[dl] \ar@<0.5ex>[dr]^{0} & \\
0 \ar@<0.5ex>[ur] & & \bb{C}\ar@<0.5ex>[ul]^{1}
}\end{aligned} \hspace{0.6in} 
\mc{F}(H): \begin{aligned}\xymatrix@=1.5pc@L=0.1pc{
\bb{C} \ar@<0.5ex>[dr]^{1} & & 0 \ar@<0.5ex>[dl] \\
 & \bb{C} \ar@<0.5ex>[ul]^{0} \ar@<0.5ex>[ur]\ar@<0.5ex>[dl]^{0} \ar@<0.5ex>[dr] & \\
\bb{C} \ar@<0.5ex>[ur]^{1} & & 0\ar@<0.5ex>[ul]
}\end{aligned}\]
Note that $P$ is a submodule of both $H$ and $\mc{F}(H)$. Due to the description in (\ref{eq:defP-I}), we obtain the injective envelope $I^P$ of $P$ by the exact sequence
\[0\lra P \lra H \oplus \mc{F}(H) \lra I^P \lra 0,\]
where the first map is the diagonal inclusion.
\end{remark}

\subsection{Description of the indecomposables}\label{subsec:indecomp}

In this section we describe the indecomposable representations of the quiver $\rep(\Q,\I)$ obtained in Theorem \ref{thm:quiver}. We consider each connected component of $\Q$. Clearly, the only indecomposable representations of the isolated vertices are the simples. 

Next, consider the quiver $\xymatrix{1 \ar@<0.5ex>[r] & 2 \ar@<0.5ex>[l]}$ with all 2-cycles zero. There are only 4 indecomposable representations (see also \cite{lor-wal}): the 2 simples  $\xymatrix{\bb{C} \ar@<0.5ex>[r] & 0 \ar@<0.5ex>[l]}$ and $\xymatrix{0 \ar@<0.5ex>[r] & \bb{C} \ar@<0.5ex>[l]}$, together with their respective projective covers $\xymatrix{\bb{C} \ar@<0.5ex>[r]^1 & \bb{C} \ar@<0.5ex>[l]^0}$ and  $\xymatrix{\bb{C} \ar@<0.5ex>[r]^0 & \bb{C} \ar@<0.5ex>[l]^1}$.

Denote by $(\tilde{\Q},\tilde{\I})$ the largest connected component of $(\Q,\I)$, that is, the quiver
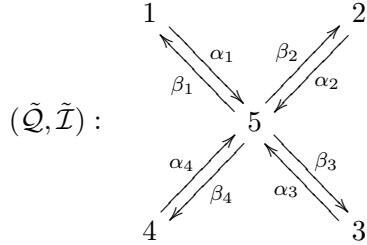
\[\hspace{-0.2in}(\tilde{\Q},\tilde{\I}): \hspace{0.1in} \begin{aligned}\xymatrix@=2.3pc@L=0.1pc{
1 \ar@<0.5ex>[dr]^{\alpha_1} & & 2 \ar@<0.5ex>[dl]^{\alpha_2} \\
 & 5 \ar@<0.5ex>[ul]^{\beta_1} \ar@<0.5ex>[ur]^{\beta_2}\ar@<0.5ex>[dl]^{\beta_4} \ar@<0.5ex>[dr]^{\beta_3} & \\
4 \ar@<0.5ex>[ur]^{\alpha_4} & & 3\ar@<0.5ex>[ul]^{\alpha_3}
}\end{aligned}\]
with relations $\tilde{\I}$ as in Theorem \ref{thm:quiver}. In contrast with the quivers obtained in \cite{lor-wal}, the quiver  $(\tilde{\Q},\tilde{\I})$ is not representation-finite, that is, the category $\rep(\tilde{\Q},\tilde{\I})$ has infinitely many indecomposable representations.  This follows by realizing the extended Dynkin quiver $\hat{D}_4$ as a subquiver $\Q(\alpha)$ (resp. $\Q(\beta)$) of $(\tilde{\Q},\tilde{\I})$ by considering only the arrows $\alpha_i$ (resp. the arrows $\beta_i$), where $i=1,2,3,4$. Nevertheless, we are able to describe the indecomposables due to the fact that $(\tilde{\Q},\tilde{\I})$ is of tame representation type (Theorem \ref{thm:tame}).

More precisely, consider the quiver $\hat{D}_4$ as in Example \ref{ex:dee4hat}. We construct two embeddings of categories
\[ \alpha: \rep(\hat{D}_4) \to \rep(\tilde{\Q},\tilde{\I}) \, \mbox{ and } \, \beta : \rep(\hat{D}_4) \to \rep(\tilde{\Q},\tilde{\I})\]
as follows. For a representation $V\in  \rep(\hat{D}_4)$, set $\alpha(V),\beta(V)$ to be the representations
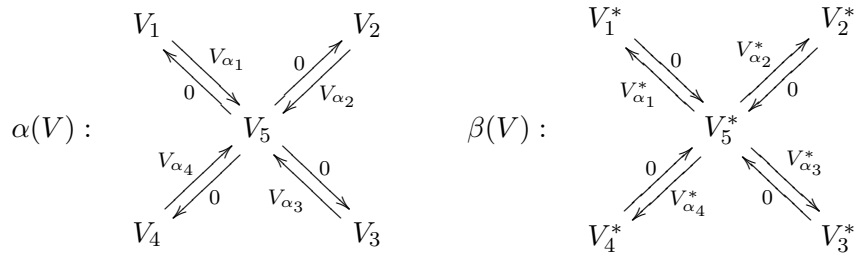
\[\alpha(V): \hspace{0.1in} \begin{aligned}\xymatrix@=2pc@L=0.1pc{
V_1 \ar@<0.5ex>[dr]^{V_{\alpha_1}} & & V_2 \ar@<0.5ex>[dl]^{V_{\alpha_2}} \\
 & V_5 \ar@<0.5ex>[ul]^{0} \ar@<0.5ex>[ur]^{0}\ar@<0.5ex>[dl]^{0} \ar@<0.5ex>[dr]^{0} & \\
V_4 \ar@<0.5ex>[ur]^{V_{\alpha_4}} & & V_3\ar@<0.5ex>[ul]^{V_{\alpha_3}}
}\end{aligned} \hspace{0.4in}
\beta(V): \hspace{0.1in} \begin{aligned}\xymatrix@=2pc@L=0.1pc{
V^*_1 \ar@<0.5ex>[dr]^{0} & & V_2^* \ar@<0.5ex>[dl]^{0} \\
 & V_5^* \ar@<0.5ex>[ul]^{V^*_{\alpha_1}} \ar@<0.5ex>[ur]^{V^*_{\alpha_2}}\ar@<0.5ex>[dl]^{V^*_{\alpha_4}} \ar@<0.5ex>[dr]^{V^*_{\alpha_3}} & \\
V^*_4 \ar@<0.5ex>[ur]^{0} & & V^*_3\ar@<0.5ex>[ul]^{0}
}\end{aligned}\]

Clearly, $\alpha,\beta$ send indecomposables to indecomposables. For example, for an arbitrary $n>0$, we have in $\rep(\tilde{\Q},\tilde{\I})$ two 1-parameter families of indecomposables $\alpha(R_n(\lambda))$ and $\beta(R_n(\lambda))$ in the dimension vector $(n,n,n,n,2n)$ induced by the representations $R_n(\lambda)\in\rep(\hat{D}_4)$ described in Example \ref{ex:dee4hat}.

Note that we have 4 projective-injective indecomposable representations of $(\tilde{\Q},\tilde{\I})$ that are not obtained from $\alpha,\beta$, namely: $\sP^1 = \sI^2$, $\sP^2 = \sI^1$, $\sP^3 = \sI^4$ and $\sP^4=\sI^3$.

As explained in Example \ref{ex:dee4hat}, the indecomposable representations of $\hat{D}_4$ are classified. Using this, we obtain the classification for $(\tilde{\Q},\tilde{\I})$:

\begin{theorem}\label{thm:tame}
Let $X$ be an indecomposable representation in $\rep(\tilde{\Q},\tilde{\I})$. Assume that $X$ is not isomorphic to either of the projective-injectives $\sP^1,\sP^2,\sP^3,\sP^4$. Then there is an indecomposable $V$ in $\rep(\hat{D}_4)$ such that $X$ is isomorphic to either $\alpha(V)$ or $\beta(V)$. In particular, the quiver $(\tilde{\Q},\tilde{\I})$ is of (domestic) tame representation type.
\end{theorem}

\begin{proof}
Due to the relations $\beta_i\alpha_i = 0$, the vertices $1,2,3,4$ of $(\tilde{\Q},\tilde{\I})$ are nodes as defined Section \ref{subsec:quivers}. After separating each of these nodes, we obtain the following quiver $(\Q',\I')$

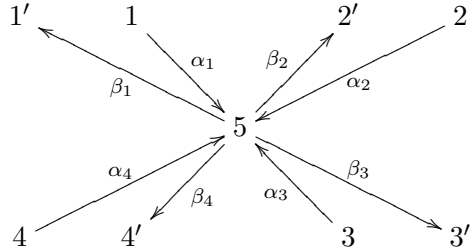
\[\xymatrix@=2.3pc@L=0.15pc{
1' & \ar[dr]^{\alpha_1} 1 & &  2' & 2 \ar[dll]^{\alpha_2}\\
 & & 5 \ar[ull]^{\beta_1} \ar[ur]^{\beta_2} \ar[dl]^{\beta_4} \ar[drr]^{\beta_3} & & \\
4 \ar[urr]^{\alpha_4} & 4' & & 3 \ar[ul]^{\alpha_3} & 3' 
}\]
with relations generated by $\alpha_i \beta_i$, where $i=1,2,3,4$, and $\alpha_1\beta_2,\alpha_1\beta_4,\alpha_2\beta_1,\alpha_2\beta_3,\alpha_3\beta_2,\alpha_3\beta_4,\alpha_4\beta_1,\alpha_4\beta_3$. By Lemma \ref{lem:node}, we can assume $X$ is an indecomposable of $(\Q',\I')$. 

Now we apply Lemma \ref{lem:addrel} four consecutive times. Since $X$ is not isomorphic to $\sP^1,\sP^2,\sP^3,\sP^4$, we obtain that the indecomposable $X$ must be a representation of the quiver $(\Q',\I'')$, where $\I''$ is obtained from $\I'$ by adding the relations $\alpha_1\beta_3, \alpha_3\beta_1,\alpha_2\beta_4,\alpha_4\beta_2$. But now the vertex $5$ is a node of $(\Q',\I'')$, and by separating the node, we obtain two copies of $\hat{D}_4$ (one copy with the middle vertex a sink and one copy with the middle vertex a source). Using Lemma \ref{lem:node} again, we obtain the result.
\end{proof}

\begin{remark}
It can be shown that the quiver $(\Q',\I')$ obtained by separating the 4 nodes has global dimension 2, and any each indecomposable $X \in \rep(\Q',\I')$, either the projective dimension or the injective dimension of $X$ is at most $1$. Hence $(\Q',\I')$ is quasi-tilted by \cite[Chapter XX, Theorem 3.10]{sim-sko3}, which suggests that an alternative approach could be used to classify its indecomposables via tilting theory (for more on quasi-tilted algebras see \cite[Section XX.3]{sim-sko3}; for more on tilting theory, see \cite[Chapter VI]{ass-sim-sko}).
\end{remark}

\begin{remark}
One can understand the morphisms between the indecomposable representations via Auslander-Reiten theory -- see \cite[Chapter VI]{ass-sim-sko} for the basic theory and relevant terminology. We can describe the Auslander-Reiten quiver of $(\tilde{\Q},\tilde{\I})$ by standard methods as follows. There are two ${\hat D}_4$ type sub-quivers of $(\tilde{\Q},\tilde{\I})$: the quiver $\Q(\alpha)$ with arrows $\alpha_i$ and the quiver $\Q(\beta)$ with arrows $\beta_j$. The Auslander-Reiten quivers for both of these sub-quivers have three parts: preprojective, preinjective and the regular one consisting of so-called tubes (see \cite[Section XIII.3]{sim-sko2}).

Now we describe the components of the Auslander-Reiten quiver of  $(\tilde{\Q},\tilde{\I})$. First, the tubes of both sub-quivers $\Q(\alpha)$ and $\Q(\beta)$ appear as separate components of the Auslander-Reiten quiver of $(\tilde{\Q},\tilde{\I})$. Next, the preinjective component of the quiver $\Q(\alpha)$ and the preprojective component of $\Q(\beta)$  are glued together along the four simple representations $\sS^1, \sS^2, \sS^3, \sS^4$ to become one component of the Auslander-Reiten quiver of $(\tilde{\Q},\tilde{\I})$. Finally, the last component of the Auslander-Reiten quiver of $(\tilde{\Q},\tilde{\I})$ is obtained by gluing together the preinjective component of $\Q(\beta)$ and the preprojective component of $\Q(\alpha)$ along the simple module $\sS^5$. The four projective-injective modules $\sP^1,\sP^2, \sP^3, \sP^4$ are also part of this component, and they occur in the middle along with $\sS^5$. They are part of four almost split sequences, described as follows. Denote by $M_{\alpha_i}$ the two dimensional module with the nonzero linear map $\alpha_i$ and by $M_{\beta_j}$ the two dimensional module with the nonzero linear map $\beta_j$. Then we have an almost split sequence
\[
\begin{matrix}&&\sP^3&&\\
&\nearrow&&\searrow&\\
M_{\beta_1}&&&&M_{\alpha_3}\\
&\searrow&&\nearrow\\
&&\sS^5&&
\end{matrix}
\]
and three more analogous ones by symmetry.
\end{remark}

\section{Some local cohomology calculations}\label{sec:loccoh}

We conclude this article with a couple of examples of local cohomology computations for equivariant $\D$-modules, with support in orbit closures. Many of these calculations are standard, but some require an understanding of the structure of the category $\opmod_{\GL}(\D_X)$. We begin with the following general observation. Suppose that $M$ is a $\GL$-equivariant $\D_X$-module with support contained in $\ol{O_i}$. It follows that for $j\geq i$ we have
\[H^k_{\ol{O_j}}(M) = \begin{cases}
 M & \mbox{if }k = 0, \\
 0 & \mbox{otherwise}.
\end{cases}\]
We will thus only be interested in studying the local cohomology groups $H^k_{\ol{O_j}}(M)$ when $j<i$. For iterated local cohomology groups of $S$ with respect to a family of $\GL$-invariant closed subsets we have the following.

\begin{theorem}\label{thm:iteratedloccoh}
The non-zero iterated local cohomology groups $H^{i_1}_{\ol{O_{j_1}}}(\cdots H^{i_k}_{\ol{O_{j_k}}}(S))$ with $j_1<\cdots<j_k$ are:
\begin{equation}\label{eq:HS}
  H^1_{\ol{O_3}}(S)=\frac{S_{\Delta}}{S},\quad  H^2_{\ol{O_2}}(S)= D_0,\quad H^4_{O_0}(S) = E,
 \end{equation}
\begin{equation}\label{eq:HHS}
  H^1_{\ol{O_2}}(H^1_{\ol{O_3}}(S))=D_0,\quad H^3_{O_0}(H^1_{\ol{O_3}}(S))= E,\quad H^2_{O_0}(H^2_{\ol{O_2}}(S))= E,
 \end{equation}
\begin{equation}\label{eq:HHHS}
  H^2_{O_0}(H^1_{\ol{O_2}}(H^1_{\ol{O_3}}(S)))=E.
\end{equation}
\end{theorem}

The formulas (\ref{eq:HHS}--\ref{eq:HHHS}) follow from (\ref{eq:HS}) and the fact that the orbit closures are cohomological complete intersections (see \cite[Example~4.2]{lyub-nmbs}). In particular, the only non-vanishing Lyubeznik number for each orbit closure is the highest one. Based on Theorem~\ref{thm:iteratedloccoh} we can also determine the local cohomology groups of simple $\GL$-equivariant $\D_X$-modules with support in the orbit closures as follows.

\begin{theorem}\label{thm:loccoh-simples}
 The only non-zero local cohomology modules $H^k_{\ol{O_j}}(M)$ when $M$ is a simple $\GL$-equivariant $\D_X$-module with support $\ol{O_i}$ and $j<i$ are the ones in (\ref{eq:HS}), as well as the following:
 \begin{equation}\label{eq:HD0}
  H^2_{O_0}(D_0)=E,
 \end{equation}
 \begin{equation}\label{eq:HP}
  H^1_{\ol{O_2}}(P)=D_0 \oplus E,\quad  H^1_{O_0}(P) = H^3_{O_0}(P) = E,
 \end{equation}
 \begin{equation}\label{eq:HQ0}
  H^1_{\ol{O_3}}(Q_0)=\frac{(Q_0)_{\Delta}}{Q_0},\quad  H^2_{\ol{O_2}}(Q_0)= H^2_{O_0}(Q_0) = E,
 \end{equation}
 \begin{equation}\label{eq:HG}
  H^1_{\ol{O_3}}(G_1) = H^1_{\ol{O_2}}(G_1) = D_1,\quad  H^1_{\ol{O_3}}(G_{-1}) = H^1_{\ol{O_2}}(G_{-1}) = D_2.
 \end{equation} 
\end{theorem}

Perhaps the most interesting computations in Theorem~\ref{thm:loccoh-simples} are that of $H^1_{\ol{O_2}}(P)=D_0$, as well as that of the local cohomology groups of $Q_0$, particularly since they require some of the structure theory that we developed for the category of $\GL$-equivariant $\D_X$-modules.

\begin{proof}[Proof of Theorem~\ref{thm:iteratedloccoh}]
 Since $\ol{O_3}$ is the hypersurface defined by $\Delta=0$ it follows that $H^i_{\ol{O_3}}(S) = S_{\Delta}/S$ for $i=0$, and it is $0$ otherwise. The calculation of $H^{\bullet}_{\ol{O_2}}(S)$ is a special case of \cite[Theorem~1.2]{raicu-veronese}, while that of $H^{\bullet}_{O_0}(S)$ is known since the ideal of $O_0$ is the maximal homogeneous ideal of $S$ \cite[Corollary~A1.6]{geom-syzygies}.
 
 For $i<j$ we have a spectral sequence $H^p_{\ol{O_i}}(H^q_{\ol{O_j}}(S)) \Rightarrow H^{p+q}_{\ol{O_i}}(S)$. We get using (\ref{eq:HS}) that this spectral sequence degenerates, since $H^q_{\ol{O_j}}(S)$ is non-zero for a unique value of $q$. This proves (\ref{eq:HHS}), as well as the vanishing of the remaining modules $H^p_{\ol{O_i}}(H^q_{\ol{O_j}}(S))$. The equation (\ref{eq:HHHS}), as well as the vanishing of $H^{i_1}_{O_0}(H^{i_2}_{\ol{O_2}}(H^{i_3}_{\ol{O_3}}(S)))$ for $(i_1,i_2,i_3)\neq(2,1,1)$ follows from (\ref{eq:HS}--\ref{eq:HHS}) since $H^1_{\ol{O_2}}(H^1_{\ol{O_3}}(S))= D_0 = H^2_{\ol{O_2}}(S)$.
\end{proof}

\begin{proof}[Proof of Theorem~\ref{thm:loccoh-simples}]
 The equation (\ref{eq:HD0}) and the vanishing of $H^i_{O_0}(D_0)$ for $i\neq 2$ follows from (\ref{eq:HS}--\ref{eq:HHS}). To prove (\ref{eq:HP}), consider the non-split exact sequence
 \begin{equation}\label{eq:PSdE}
 0 \lra P \lra S_{\Delta}/S \lra E \lra 0
 \end{equation}
 coming from Corollary~\ref{cor:factors-Sdelta} and Lemma~\ref{lem:indS-del}. Since $E$ is supported at $O_0$ we have $H^i_{\ol{O_j}}(E) = E$ for $i=0$ and it is $0$ otherwise, while the local cohomology modules of $S_{\Delta}/S = H^1_{\ol{O_3}}(S)$ have been computed in (\ref{eq:HHS}). The long exact sequence obtained by applying $H^0_{\ol{O_2}}(\bullet)$ to (\ref{eq:PSdE}) yields then a short exact sequence
 \[0\lra E \lra H^1_{\ol{O_2}}(P) \lra D_0 \lra 0\]
 and the vanishing of $H^i_{\ol{O_2}}(P)$ for $i\neq 1$. By Theorem~\ref{thm:quiver}, any extension between $D_0$ and $E$ is split, hence $H^1_{\ol{O_2}}(P) = D_0 \oplus E$. We next consider the long exact sequence obtained by applying $H^0_{O_0}(\bullet)$ to (\ref{eq:PSdE}) and get
 \[ E = H^0_{O_0}(E) \simeq H^1_{O_0}(P),\quad H^3_{O_0}(P) \simeq H^3_{O_0}(S_{\Delta}/S) = E,\]
 and $H^i_{O_0}(P)=0$ for $i\neq 1,3$, proving (\ref{eq:HP}).
 
 The calculation of $H^i_{\ol{O_3}}(Q_0)$ follows from the fact that $\ol{O_3}$ is defined by $\Delta=0$. Since the only non-vanishing is obtained for $i=1$, the spectral sequence $H^p_{\ol{O_2}}(H^q_{\ol{O_3}}(Q_0)) \Rightarrow H^{p+q}_{\ol{O_2}}(Q_0)$ degenerates and we get $H^i_{\ol{O_2}}(Q_0) = H^{i-1}_{\ol{O_2}}(H^1_{\ol{O_3}}(Q_0))$. We have from Corollary~\ref{cor:factors-FD0-delta} and Lemma~\ref{lem:indQ0-del} a non-split exact sequence
 \[0\lra P \lra H^1_{\ol{O_3}}(Q_0) \lra D_0 \lra 0,\]
 which yields by applying $H^0_{\ol{O_2}}(\bullet)$, and using (\ref{eq:HP}) and the fact that $D_0$ has support in $\ol{O_2}$, an exact sequence
 \[0\lra H^0_{\ol{O_2}}(H^1_{\ol{O_3}}(Q_0)) \overset{\a}{\lra} D_0 \overset{\b}{\lra} D_0\oplus E \lra H^1_{\ol{O_2}}(H^1_{\ol{O_3}}(Q_0)) \lra 0,\]
 as well as the vanishing $H^i_{\ol{O_2}}(H^1_{\ol{O_3}}(Q_0))=0$ for $i\geq 2$. If the map $\a$ is non-zero, then it must be an isomorphism since $D_0$ is simple, which would show that $D_0$ is a submodule of $H^1_{\ol{O_3}}(Q_0)=(Q_0)_{\Delta}/Q_0$ and contradict Lemma~\ref{lem:indQ0-del}. It follows that $\a=0$, hence $H^1_{\ol{O_2}}(Q_0)=H^0_{\ol{O_2}}(H^1_{\ol{O_3}}(Q_0)) = 0$, and moreover we have $\coker(\b)=E$ which implies $H^2_{\ol{O_2}}(Q_0)=H^1_{\ol{O_2}}(H^1_{\ol{O_3}}(Q_0)) = E$. To finish the proof of (\ref{eq:HQ0}) we note that $H^q_{\ol{O_2}}(Q_0)$ is non-zero only for $q=2$, and therefore it yields the degeneration of the spectral sequence $H^p_{O_0}(H^q_{\ol{O_2}}(Q_0)) \Rightarrow H^{p+q}_{O_0}(Q_0)$. We get that $H^q_{\ol{O_2}}(Q_0)=E$ for $q=2$ and is $0$ otherwise. 
 
 Since $F_1/G_1 = D_1$ we have $(G_1)_{\Delta} = (F_1)_{\Delta} = F_1$ and therefore
 \[H^1_{\ol{O_3}}(G_1) = (G_1)_{\Delta}/G_1 = F_1/G_1 = D_1\]
 is the only non-vanishing group $H^i_{\ol{O_3}}(G_1)$. The standard (by now) spectral sequence argument shows that $H^i_{\ol{O_2}}(G_1) = D_1$ for $i=1$ and is $0$ otherwise. A similar argument, based on $F_{-1}/G_{-1}=D_2$, applies to show the rest of (\ref{eq:HG}). To conclude the determination of the local cohomology modules of $G_{\pm 1}$ it remains to prove that $H^i_{O_0}(G_{\pm 1})=0$ for all $i$, which in turn is a consequence of the fact that $H^i_{O_0}(D_j)=0$ for all~$i$ and $j=1,2$. This vanishing result follows since on the one hand the modules $H^i_{O_0}(D_j)$ are direct sums of copies of $E$ (beging supported on $O_0$), and on the other hand their characters may only involve terms of the form~$e^{\ll}$ with $|\ll| \equiv 1,2\ \mbox{(mod 3)}$ by (\ref{eq:wts-D012}). Comparing this with (\ref{eq:character-E}) we get the desired vanishing.
 
 Finally, if $M$ is any of the modules $G_2,G_3,G_4,Q_1,Q_2$ we have that $M_{\Delta}=M$, so $H^i_{\ol{O_3}}(M)=0$ for all~$i$. Using the spectral sequence, this shows that $H^i_{\ol{O_2}}(M)=H^i_{O_0}(M)=0$ for all $i$, concluding the proof.
\end{proof}

\section*{Acknowledgements}
Experiments with the computer algebra software Macaulay2 \cite{M2} have provided numerous valuable insights. Raicu acknowledges the support of the Alfred P. Sloan Foundation, and of the National Science Foundation Grant No.~1600765. Weyman acknowledges partial support of the Sidney Professorial Fund and of the National Science Foundation grant No.~1400740.

\end{document}